\documentclass[12pt,reqno,draft]{amsart}
\usepackage{amssymb,amscd,url}

\begin{document}

\allowdisplaybreaks
\hyphenation{non-archi-me-dean}


\newtheorem{theorem}{Theorem}
\newtheorem{lemma}[theorem]{Lemma}
\newtheorem*{unnumberedlemma}{Lemma}
\newtheorem{conjecture}[theorem]{Conjecture}
\newtheorem{proposition}[theorem]{Proposition}
\newtheorem{corollary}[theorem]{Corollary}
\newtheorem*{claim}{Claim}

\theoremstyle{definition}
\newtheorem*{definition}{Definition}
\newtheorem{remark}[theorem]{Remark}
\newtheorem{example}[theorem]{Example}
\newtheorem{question}[theorem]{Question}

\theoremstyle{remark}
\newtheorem*{acknowledgement}{Acknowledgements}


\newenvironment{notation}[0]{%
  \begin{list}%
    {}%
    {\setlength{\itemindent}{0pt}
     \setlength{\labelwidth}{4\parindent}
     \setlength{\labelsep}{\parindent}
     \setlength{\leftmargin}{5\parindent}
     \setlength{\itemsep}{0pt}
     }%
   }%
  {\end{list}}

\newenvironment{parts}[0]{%
  \begin{list}{}%
    {\setlength{\itemindent}{0pt}
     \setlength{\labelwidth}{1.5\parindent}
     \setlength{\labelsep}{.5\parindent}
     \setlength{\leftmargin}{2\parindent}
     \setlength{\itemsep}{0pt}
     }%
   }%
  {\end{list}}
\newcommand{\Part}[1]{\item[\upshape#1]}

\renewcommand{\a}{\alpha}
\newcommand{\aupper}{\overline{\alpha}}
\newcommand{\alower}{\underline{\alpha}}
\renewcommand{\b}{\beta}
\newcommand{\bfbeta}{{\boldsymbol{\beta}}}
\newcommand{\g}{\gamma}
\renewcommand{\d}{\delta}
\newcommand{\e}{\epsilon}
\newcommand{\bfepsilon}{\boldsymbol{\epsilon}}
\newcommand{\f}{\varphi}
\newcommand{\bfphi}{{\boldsymbol{\f}}}
\renewcommand{\l}{\lambda}
\newcommand{\bfl}{{\boldsymbol{\lambda}}}
\renewcommand{\k}{\kappa}
\newcommand{\lhat}{\hat\lambda}
\newcommand{\m}{\mu}
\newcommand{\bfmu}{{\boldsymbol{\mu}}}
\renewcommand{\o}{\omega}
\renewcommand{\r}{\rho}
\newcommand{\rbar}{{\bar\rho}}
\newcommand{\s}{\sigma}
\newcommand{\sbar}{{\bar\sigma}}

\renewcommand{\t}{\tau}
\newcommand{\z}{\zeta}

\newcommand{\D}{\Delta}
\newcommand{\G}{\Gamma}
\newcommand{\F}{\Phi}
\renewcommand{\L}{\Lambda}
\newcommand{\ga}{{\mathfrak{a}}}
\newcommand{\gb}{{\mathfrak{b}}}
\newcommand{\gn}{{\mathfrak{n}}}
\newcommand{\gp}{{\mathfrak{p}}}
\newcommand{\gP}{{\mathfrak{P}}}
\newcommand{\gq}{{\mathfrak{q}}}

\newcommand{\Abar}{{\bar A}}
\newcommand{\Ebar}{{\bar E}}
\newcommand{\Kbar}{{\bar K}}
\newcommand{\Pbar}{{\bar P}}
\newcommand{\Sbar}{{\bar S}}
\newcommand{\Tbar}{{\bar T}}
\newcommand{\ybar}{{\bar y}}
\newcommand{\phibar}{{\bar\f}}

\newcommand{\ftilde}{{\tilde f}}

\newcommand{\Acal}{{\mathcal A}}
\newcommand{\Bcal}{{\mathcal B}}
\newcommand{\Ccal}{{\mathcal C}}
\newcommand{\Dcal}{{\mathcal D}}
\newcommand{\Ecal}{{\mathcal E}}
\newcommand{\Fcal}{{\mathcal F}}
\newcommand{\Gcal}{{\mathcal G}}
\newcommand{\Hcal}{{\mathcal H}}
\newcommand{\Ical}{{\mathcal I}}
\newcommand{\Jcal}{{\mathcal J}}
\newcommand{\Kcal}{{\mathcal K}}
\newcommand{\Lcal}{{\mathcal L}}
\newcommand{\Mcal}{{\mathcal M}}
\newcommand{\Ncal}{{\mathcal N}}
\newcommand{\Ocal}{{\mathcal O}}
\newcommand{\Pcal}{{\mathcal P}}
\newcommand{\Qcal}{{\mathcal Q}}
\newcommand{\Rcal}{{\mathcal R}}
\newcommand{\Scal}{{\mathcal S}}
\newcommand{\Tcal}{{\mathcal T}}
\newcommand{\Ucal}{{\mathcal U}}
\newcommand{\Vcal}{{\mathcal V}}
\newcommand{\Wcal}{{\mathcal W}}
\newcommand{\Xcal}{{\mathcal X}}
\newcommand{\Ycal}{{\mathcal Y}}
\newcommand{\Zcal}{{\mathcal Z}}

\renewcommand{\AA}{\mathbb{A}}
\newcommand{\BB}{\mathbb{B}}
\newcommand{\CC}{\mathbb{C}}
\newcommand{\FF}{\mathbb{F}}
\newcommand{\GG}{\mathbb{G}}
\newcommand{\NN}{\mathbb{N}}
\newcommand{\PP}{\mathbb{P}}
\newcommand{\QQ}{\mathbb{Q}}
\newcommand{\RR}{\mathbb{R}}
\newcommand{\ZZ}{\mathbb{Z}}

\newcommand{\bfa}{{\boldsymbol a}}
\newcommand{\bfb}{{\boldsymbol b}}
\newcommand{\bfc}{{\boldsymbol c}}
\newcommand{\bfe}{{\boldsymbol e}}
\newcommand{\bff}{{\boldsymbol f}}
\newcommand{\bfg}{{\boldsymbol g}}
\newcommand{\bfh}{{\boldsymbol h}}
\newcommand{\bfp}{{\boldsymbol p}}
\newcommand{\bfr}{{\boldsymbol r}}
\newcommand{\bfs}{{\boldsymbol s}}
\newcommand{\bft}{{\boldsymbol t}}
\newcommand{\bfu}{{\boldsymbol u}}
\newcommand{\bfv}{{\boldsymbol v}}
\newcommand{\bfw}{{\boldsymbol w}}
\newcommand{\bfx}{{\boldsymbol x}}
\newcommand{\bfy}{{\boldsymbol y}}
\newcommand{\bfz}{{\boldsymbol z}}
\newcommand{\bfA}{{\boldsymbol A}}
\newcommand{\bfC}{{\boldsymbol C}}
\newcommand{\bfF}{{\boldsymbol F}}
\newcommand{\bfB}{{\boldsymbol B}}
\newcommand{\bfD}{{\boldsymbol D}}
\newcommand{\bfE}{{\boldsymbol E}}
\newcommand{\bfG}{{\boldsymbol G}}
\newcommand{\bfI}{{\boldsymbol I}}
\newcommand{\bfM}{{\boldsymbol M}}
\newcommand{\bfzero}{{\boldsymbol{0}}}

\newcommand{\Amp}{\operatorname{Amp}}
\newcommand{\Aut}{\operatorname{Aut}}
\newcommand{\bad}{\textup{bad}}
\newcommand{\Disc}{\operatorname{Disc}}
\newcommand{\dist}{\Delta}  
\newcommand{\Div}{\operatorname{Div}}
\newcommand{\End}{\operatorname{End}}
\newcommand{\Eff}{\operatorname{Eff}}
\newcommand{\Family}{{\mathcal A}}  
\newcommand{\Fatou}{{\mathcal F}}
\newcommand{\Fbar}{{\bar{F}}}
\newcommand{\Fix}{\operatorname{Fix}}
\newcommand{\Gal}{\operatorname{Gal}}
\newcommand{\GL}{\operatorname{GL}}
\newcommand{\good}{\textup{good}}
\newcommand{\Index}{\operatorname{Index}}
\newcommand{\Image}{\operatorname{Image}}
\newcommand{\interior}{\operatorname{int}}
\newcommand{\Julia}{{\mathcal J}}
\newcommand{\liftable}{{\textup{liftable}}}
\newcommand{\hhat}{{\hat h}}
\newcommand{\hhatlower}{\underline{\hat h}}
\newcommand{\hhatplus}{{\hat h^{\scriptscriptstyle+}}}
\newcommand{\hhatminus}{{\hat h^{\scriptscriptstyle-}}}
\newcommand{\hhatpm}{{\hat h^{\scriptscriptstyle\pm}}}
\newcommand{\hplus}{h^{\scriptscriptstyle+}}
\newcommand{\Ker}{{\operatorname{ker}}}
\newcommand{\Lift}{\operatorname{Lift}}
\newcommand{\MOD}[1]{~(\textup{mod}~#1)}
\newcommand{\Nef}{\operatorname{Nef}}
\newcommand{\Norm}{{\operatorname{\mathsf{N}}}}
\newcommand{\notdivide}{\nmid}
\newcommand{\normalsubgroup}{\triangleleft}
\newcommand{\NS}{{\operatorname{NS}}}
\newcommand{\ns}{{\textup{sm}}} 
\newcommand{\odd}{{\operatorname{odd}}}
\newcommand{\onto}{\twoheadrightarrow}
\newcommand{\ord}{\operatorname{ord}}
\newcommand{\Orbit}{{\mathcal O}}
\newcommand{\PGL}{\operatorname{PGL}}
\newcommand{\Pic}{\operatorname{Pic}}
\newcommand{\Prob}{\operatorname{Prob}}
\newcommand{\psef}{\textup{psef}}
\newcommand{\Qbar}{{\bar{\QQ}}}
\newcommand{\rank}{\operatorname{rank}}
\newcommand{\Res}{{\operatorname{Res}}}
\newcommand{\Resultant}{\operatorname{Res}}
\newcommand{\rest}[2]{\left.{#1}\right\vert_{{#2}}}  
\renewcommand{\setminus}{\smallsetminus}
\newcommand{\Span}{\operatorname{Span}}
\newcommand{\Spec}{{\operatorname{Spec}}}
\newcommand{\Supp}{\operatorname{Supp}}
\newcommand{\tors}{{\textup{tors}}}
\newcommand{\Trace}{\operatorname{Trace}}
\newcommand{\UHP}{{\mathfrak{h}}}    

\newcommand{\longhookrightarrow}{\lhook\joinrel\longrightarrow}
\newcommand{\longonto}{\relbar\joinrel\twoheadrightarrow}


\title[Dynamical and arithmetic degrees of rational maps]
{On the dynamical and arithmetic degrees
of rational self-maps of algebraic varieties}
\date{March 2013}

\author[Shu Kawaguchi and Joseph H. Silverman]
  {Shu Kawaguchi and Joseph H. Silverman}
\email{kawaguch@math.kyoto-u.ac.jp, jhs@math.brown.edu}
\address{Department of Mathematics, Faculty of Science, Kyoto University, 
Kyoto, 606-8502, Japan}
\address{Mathematics Department, Box 1917
         Brown University, Providence, RI 02912 USA}
\subjclass{Primary: 37P15; Secondary: 37P05, 37P30, 37P55, 11G50}
\keywords{arithmetic degree; dynamical degree; canonical height}
\thanks{The first author's research supported by JSPS
grant-in-aid for young scientists (B) 24740015.
The second author's research supported by NSF DMS-0854755
and Simons Collaboration Grant \#241309.}

\begin{abstract}
Let $f:X\dashrightarrow X$ be a dominant rational map of a smooth
projective variety defined over a characteristic~$0$ global field~$K$,
let~$\d_f$ be the dynamical degree of~$f$, and let
$h_X:X(\Kbar)\to[1,\infty)$ be a Weil height relative to an ample
divisor. We prove that for every~$\e>0$ there is a height bound
\[
  h_X\circ f^n \ll (\d_f+\e)^n h_X,
\]
valid for all points whose~$f$-orbit is well-defined,
where the implied constant depends only on~$X$,~$h_X$,~$f$, and~$\e$.
An immediate corollary is a fundamental inequality
$\aupper_f(P)\le\d_f$ for the upper arithmetic degree.  If further~$f$
is a morphism and~$D$ is a divisor satisfying an algebraic equivalence
$f^*D\equiv\b D$ for some $\b>\sqrt{\d_f}$, we prove that the
canonical height $\hhat_{f,D}=\lim \b^{-n}h_D\circ f^n$ converges and
satisfies $\hhat_{f,D}\circ f=\b\hhat_{f,D}$ and
$\hhat_{f,D}=h_D+O\left(\sqrt{h_X}\right)$.  We conjecture that
$\aupper_f(P)=\d_f$ whenever the $f$-orbit of~$P$ is Zariski dense
and describe some cases for which we can prove our conjecture.
\end{abstract}

\maketitle

\newpage 

\section*{Introduction}
\label{section:intro}

Let $X/\CC$ be a smooth projective variety, and let
$f:X\dashrightarrow X$ be a dominant rational map.  The dynamical
degree of~$f$ is a measure of the geometric complexity of the
iterates~$f^n$ of~$f$. More precisely, it measures the complexity of
the induced maps~$(f^n)^*$ of the iterates of~$f$ on the
N\'eron-Severi group~$\NS(X)_\RR$ of~$X$,\footnote{We
  write~$\NS(X)_\RR$ for $\NS(X)\otimes\RR$, and similarly
  for~$\NS(X)_\QQ$ and~$\NS(X)_\CC$.}  where we note that in
general~$(f^n)^*$ need not be equal to~$(f^*)^n$.

\begin{definition}
Let~$X/\CC$ be a (smooth) projective variety and let
$f:X\dashrightarrow X$ be a dominant rational map as above.  The
\emph{dynamical degree of~$f$} is
\[ 
  \d_f = \lim_{n\to\infty} \rho\bigl((f^n)^*,\NS(X)_\RR\bigr)^{1/n},
\]
where in general~$\rho(A,V)$ denotes the spectral radius of a linear
transformation $A:V\to V$ of a real or complex vector space.  The
limit defining~$\d_f$ converges and is a birational invariant, so in
particular there is no need to assume that~$X$ is smooth;
see~\cite[Proposition~1.2(iii)]{MR2179389},
Remark~\ref{remark:dfviaintersection}, and
Corollary~\ref{corollary:dfexists}.
\end{definition}


The study of the dynamical degree and its relation to entropy was
initiated in~\cite{MR1139553,MR1488341} and is currently an area of
active research; see for example \cite{MR1732080, MR2111418,
  MR2280499, arxiv0512507, MR2449533, MR2428100, MR1704282, MR1793690,
  arxiv0608267, arxiv1011.2854, MR2358970, arxiv1007.0253,
  arxiv1010.6285, MR1836434, MR2273670, MR2497321}. In this article we
describe how the geometrically defined dynamical degree of a map
limits the arithmetic complexity of its orbits,
and we prove an inequality relating the dynamical degree to an analogous
arithmetic degree defined in~\cite{arxiv1111.5664}.
\par
Before stating our main results, we set some notation that will be
used throughout this article.  
\begin{notation}
\item[$K$] 
  Either a number field or a one-dimensional function field
  of characteristic~$0$. We let~$\Kbar$ be an algebraic closure of~$K$.
\item[$X,f/K$] 
   Either $X$ is a smooth projective variety and $f:X\dashrightarrow
   X$ is a dominant rational map, all defined over $K$; or $X$ is a
   normal projective variety and $f:X\to X$ is a dominant morphism, all
   defined over $K$.  (See also Remark~\ref{remark:singvar}.)
\item[$h_X$]
   An (absolute logarithmic) Weil height $h_X:X(\Kbar)\to[0,\infty)$
   relative to an ample divisor.
\item[$\hplus_X$]
  For convenience, we set
  $\hplus_X(P) = \max\bigl\{h_X(P),1\bigr\}$.
\item[$\Orbit_f(P)$]
    The (\emph{forward}) \emph{$f$-orbit} of~$P$, i.e., 
         $\Orbit_f(P)=\{f^n(P):n\ge0\}$.
\item[$I_f$]
   The \emph{indeterminacy locus} of~$f$, i.e., the set of
   points at which~$f$ is not well-defined.
\item[$X_f(\Kbar)$]
   The set of points $P\in X(\Kbar)$ whose forward orbit~$\Orbit_f(P)$ is
   well-defined,
   i.e., such that~$f^n(P)\notin I_f$ for all~$n\ge0$.  
   We note that~$X_f(\Kbar)$ always contains many
   points; see~\cite{MR2784670}.
\end{notation}

We refer the reader to~\cite{bombierigubler,
  hindrysilverman:diophantinegeometry, lang:diophantinegeometry,
  MR2514094} for basic definitions and properties of
Weil height functions.

Our main theorem gives a uniform upper bound for the growth of points
in orbits.

\begin{theorem}
\label{theorem:hXfnlldfenhX}
Fix $\e>0$. Then there is a constant $C=C(X,h_X,f,\e)$ so that for
all~$n\ge0$ and all~$P\in X_f(\Kbar)$,
\[
  \hplus_X\bigl(f^n(P)\bigr) \le C (\d_f+\e)^n\hplus_X(P).
\]
\end{theorem}

For rational maps~$f:\PP^N\dashrightarrow\PP^N$ of projective space,
Theorem~\ref{theorem:hXfnlldfenhX} was essentially proven
in~\cite[Proposition~13]{arxiv1111.5664}. 
The same proof works, \emph{mutatis mutandis}, for varieties
satisfying~$\Pic(X)_\RR=\RR$, and, with a little more work, for
varieties satisfying~$\NS(X)_\RR=\RR$.  But if~$\NS(X)_\RR$ has
dimension greater than~$1$, then the proof of
Theorem~\ref{theorem:hXfnlldfenhX}, which we give in
Section~\ref{section:afpledf} after several sections of preliminary
results, is considerably more intricate.

We next consider the arithmetic degree of a map at a point, as
introduced in~\cite{arxiv1111.5664}. We recall the relevant
definitions, give an elementary counting result, and then describe an
inequality that was a primary motivation for the research that led
to this paper.

\begin{definition} 
Let~$P\in X_f(\Kbar)$. The \emph{arithmetic degree of~$f$ at~$P$} is
the quantity
\[
  \a_f(P) = \lim_{n\to\infty} \hplus_X\bigl(f^n(P)\bigr)^{1/n},
\]
assuming that the limit exists. 
\end{definition} 

The arithmetic degree of~$f$ at~$P$ measures the growth rate
of the height~$h_X\bigl(f^n(P)\bigr)$ as $n\to\infty$. It is thus a
measure of the arithmetic complexity of the~$f$-orbit of~$P$.

\begin{conjecture}
\label{conjecture:afPlimitexists}
The limit defining~$\a_f(P)$ exists for all $P\in X_f(\Kbar)$.
\end{conjecture}

One reason for studying the arithmetic degree is that it determines
the height counting function for points in orbits, as in the following
elementary result, which we prove in Section~\ref{section:defs}.

\begin{proposition}
\label{proposition:limBQhXQleB}
Let~$P\in X_f(\Kbar)$ be a wandering point, i.e., a point whose
orbit~$\#\Orbit_f(P)$ is infinite. Assume further that the arithmetic
degree~$\a_f(P)$ exists.  Then
\begin{equation}
  \label{eqn:limBQhXQleB}
  \lim_{B\to\infty} \frac{\#\bigl\{Q\in\Orbit_f(P) : h_X(Q)\le B\bigr\}}{\log B}
   = \frac{1}{\log\a_f(P)}.
\end{equation}
\textup(If~$\a_f(P)=1$, then~\eqref{eqn:limBQhXQleB} is to be
read as saying that the limit is equal to~$\infty$.\textup)
\end{proposition}

\begin{definition} 
Since for the moment we lack a proof of
Conjecture~\ref{conjecture:afPlimitexists}, we define \emph{upper and
lower arithmetic degrees},
\[
  \aupper_f(P) = \limsup_{n\to\infty} \hplus_X\bigl(f^n(P)\bigr)^{1/n}
  \quad\text{and}\quad
  \alower_f(P) = \liminf_{n\to\infty} \hplus_X\bigl(f^n(P)\bigr)^{1/n}.
\]
\end{definition}

As a corollary to Theorem~\ref{theorem:hXfnlldfenhX}, we obtain the
following fundamental inequality relating the dynamical degree and the
(upper) arithmetic degree. This inequality quantifies the statement
that the arithmetic complexity of the $f$-orbit of an algebraic
point~$P$ never exceeds the geometrical-dynamical complexity of the
map~$f$.

\begin{theorem}
\label{theorem:afPledfintro}
Let $P\in X_f(\Kbar)$. Then
\[
  \aupper_f(P)\le \d_f.
\]
\end{theorem}

Classically, a polarized dynamical system is a triple~$(X,f,D)$
consisting of a morphism~$f:X\to X$ and a divisor~$D$ satisfying a
\emph{linear equivalence} $f^*D\sim\b D$ for some $\b>1$.  (Often the
definition also includes the condition that~$D$ be ample;
cf.\ \cite{zhang:distalgdyn}.)  There is a well-known theory of
canonical heights associated to polarized dynamical systems; see for
example~\cite{callsilv:htonvariety}.  Using
Theorem~\ref{theorem:hXfnlldfenhX}, we are able to partially
generalize this theory to cover the case that the relation
$f^*D\equiv\b D$ is only an \emph{algebraic equivalence}.

\begin{theorem} 
\label{theorem:hDfPbnhDfnP}
Assume that~$f:X\to X$ is a morphism, and let $D\in\Div(X)_\RR$ 
be a divisor that satisfies an algebraic equivalence
\[
  f^*D\equiv\b D
  \quad\text{for some real number $\b>\sqrt{\d_f}$,}
\]
where $\equiv$ denotes equivalence in~$\NS(X)_\RR$.
\begin{parts}
\Part{(a)}
For all $P\in X(\Kbar)$, the following limit converges\textup:
\[
  \hhat_{D,f}(P)=\lim_{n\to\infty} \b^{-n}h_D\bigl(f^n(P)\bigr).
\]
\Part{(b)}
The canonical height $\hhat_{D,f}$ in \textup{(a)} satisfies
\[
  \hhat_{D,f}\bigl(f(P)\bigr) = \b \hhat_{D,f}(P)
  \quad\text{and}\quad
  \hhat_{D,f}(P) = h_D(P) + 
     \smash{O\left(\sqrt{\hplus_X(P)}\right)}.
\]
\Part{(c)}
If $\hhat_{D,f}(P)\ne0$, then $\alower_f(P)\ge\b$.
\Part{(d)}
If~$\hhat_{D,f}(P)\ne0$ and~$\b=\d_f$, then $\a_f(P)=\d_f$.
\Part{(e)}
Assume that~$D$ is ample and that~$K$ is a number field.  Then
\[
  \hhat_{D,f}(P)=0 \quad\Longleftrightarrow\quad
  \text{$P$ is preperiodic.}
\]
\end{parts}
\end{theorem}

We note that not every morphism $f:X\to X$ admits a polarization (for
linear equivalence), but that there always exists at least one
non-zero nef divisor~$D\in\Div(X)_\RR$ satisfying $f^*D\equiv\d_fD$;
see Remark~\ref{remark:fDdfDexists}.  Hence every morphism~$f$ of
positive algebraic entropy, i.e., with dynamical degree satisfying
$\d_f>1$, admits a canonical height associated to a nef divisor.

Theorem~\ref{theorem:afPledfintro} raises a natural question: Under
what conditions is $\a_f(P)$ equal to~$\d_f$, i.e., when does the
arithmetic complexity of the $f$-orbit of a point~$P$ fully capture
the geometrical-dynamical complexity of~$f$?  This leads to the
following multi-part conjecture, into which we have incorporated
Conjecture~\ref{conjecture:afPlimitexists}, as well as an integrality
conjecture suggested by a classical conjecture~\cite{MR1704282} on the
integrality of~$\d_f$.  See also~\cite[Conjecture~42]{arxiv1111.5664},
in which~(b),~(c), and~(d) were conjectured for~$\aupper_f(P)$.


\begin{conjecture}
\label{conjecture:afPeqdfY}
Let~$P\in X_f(\Kbar)$.
\begin{parts}
\Part{(a)}
The limit defining~$\a_f(P)$ exists.
\Part{(b)}
$\a_f(P)$ is an algebraic integer.
\Part{(c)}
The collection of arithmetic degrees $\bigl\{\a_f(Q):Q\in
X_f(\Kbar)\bigr\}$ is a finite set.
\Part{(d)}
If the forward orbit~$\Orbit_f(P)$ is Zariski dense in~$X$, then
$\a_f(P)=\d_f$.
\end{parts}
\end{conjecture}

In the final section of this paper we briefly indicate some cases for
which we can prove Conjecture~\ref{conjecture:afPeqdfY}. These include
morphisms~$f$ when~$\NS(X)_\RR=\RR$, regular affine automorphisms,
surface automorphisms, and monomial maps. The proofs of these results,
together with other cases for which we can prove the weaker statement
that $\a_f(P)=\d_f(X)$ for a Zariski dense set of points~$P\in
X_f(\Kbar)$ having disjoint orbits, will appear in a companion
publication~\cite{kawsilv:dfeqafPexamples}.  See
also~\cite{kawsilv:jordanblock} for a proof of
Conjecture~\ref{conjecture:afPeqdfY}(a,b,c) when~$f$ is a morphism
and~(d) when~$f$ is an endomorphism of an abelian variety.

\begin{acknowledgement}
The authors would like to thank ICERM for providing a stimulating
research environment during their spring 2012 visits, as well as
the organizers of conferences on
Automorphisms (Shirahama 2011), Algebraic Dynamics (Berkeley 2012),
and the SzpiroFest (CUNY 2012), during which some of this research was
done. The authors would also like to thank 
Najmuddin Fakhruddin for his helpful comments and suggestions regarding
an earlier version of this article, including pointing out that our
original formulation of the main theorem was too general; see
Remark~\ref{remark:singvar} for details.
\end{acknowledgement}

\section{Some Brief Remarks}
\label{section:remarks}
In this section we make some brief remarks about dynamical degrees, 
arithmetic degrees, and canonical heights.

\begin{remark}
The assumption in Conjecture~\ref{conjecture:afPeqdfY}(d)
that~$\Orbit_f(P)$ be Zariski dense is not as strong as it
appears. This is because~$f$ induces a rational map on the Zariski
closure~$Y=\overline{\Orbit_f(P)}\subset X$ of the orbit.  So ignoring
the smoothness condition, we can apply
Conjecture~\ref{conjecture:afPeqdfY} to~$f|_Y$ and~$P\in Y_f(\Kbar)$
to deduce that~$\a_f(P)=\d_{f|_Y}$. Note that~$\a_f(P)$ is independent
of whether we view~$P$ as a point of~$X$ or a point of~$Y$, since the
restriction to~$Y$ of an ample height function~$h_X$ on~$X$ gives an
ample height function on~$Y$. 
\end{remark}

\begin{remark} 
Bellon and Viallet~\cite{MR1704282} conjecture that~$\d_f$ is an
algebraic integer. Assuming this and
Conjecture~\ref{conjecture:afPeqdfY}(d), one can more-or-less reduce
Conjectures~\ref{conjecture:afPeqdfY}(b,c) to the study of the values
of~$\d_f$ on the $f$-invariant subvarieties of~$X$.
\end{remark}

\begin{remark}
\label{remark:dfviaintersection}
Let~$H$ be an ample divisor on~$X$, and let $N=\dim(X)$.
Then~\cite[Proposition~1.2(iii)]{MR2179389} says that
\[
  \lim_{n\to\infty} \bigl((f^n)^*H\cdot H^{N-1}\bigr)^{1/n}
  = 
  \limsup_{n\to\infty} \rho\bigl((f^n)^*,\NS(X)_\RR)\bigr)^{1/n}.
\]
(Notice the right-hand side is a limsup.)  We will prove below
(Corollary~\ref{corollary:dfexists}) that the limit $\lim_{n\to\infty}
\rho\bigl((f^n)^*,\NS(X)_\RR)\bigr)^{1/n}$ exists, justifying our
definition of~$\d_f$ in terms of the action of~$(f^n)^*$
on~$\NS(X)_\RR$, but we note that the alternative definition of~$\d_f$
using intersection is more common and often more useful.
\end{remark}

\begin{remark}
\label{remark:singvar}
We have restricted our variety~$X$ to be smooth when~$f$ is not a
morphism. In our original formulation, we had only assumed that~$X$ is
normal. We thank Najmuddin Fakhruddin for pointing out that some
conditions are necessary to define the pull-back~$f^*$ on $\NS(X)_\RR$
for a dominant rational map $f:X\dashrightarrow X$. Fakhruddin has
indicated that it should suffice to take~$X$ to be $\QQ$-factorial. We
use the Lefschetz hyperplane theorem in the proof of
Lemma~\ref{lemma:DEffXHAmp}, but for a singular variety, one can use a
version of the Lefschetz hyperplane theorem~\cite[Theorem on
  page~153]{MR932724} for a general member of the linear system of a
very ample divisor.  Alternatively, if the orbit~$\Orbit_f(P)$ of~$P$
lies within the smooth locus~$X^\ns$ of~$X$, as is often the case,
then one can simply replace~$X$ with a smooth model of a projective
closure of~$X^\ns$ and reduce to the smooth case.
\end{remark}

\begin{remark}
In~\cite{MR1704282} the authors define the \emph{algebraic entropy
  of~$f$} to be the quantity $\log \d_f$.  It is thus tempting to
call $\log\a_f(P)$ the \emph{arithmetic entropy of~$(f,P)$}, and
indeed one can reformulate the definitions of~$\log\d_f$
and~$\log\a_f(P)$ to more closely resemble classical defininitions of
entropy. 
More generally, the $p^{\text{th}}$-dynamical degree $\d_p(f)$ may be
defined as the limiting
value of $\bigl((f^n)^*H^p\cdot H^{N-p}\bigr)^{1/n}$;
see~\cite[Corollaire~7]{MR2180409}.  Then $\log \d_p(f)$ is called the
$p^{\text{th}}$-algebraic entropy of~$f$. One might use Arakelov
intersection theory to similarly define higher codimension arithemtic
entropies for self-maps of arithmetic varieties.
\end{remark} 

\begin{remark}
We use~$\hplus_X$ instead of~$h_X$ in the definition of arithmetic
degree simply to ensure that~$\alower_f(P)\ge1$, even in the rare
situation that~$P$ is periodic and~$h_X\bigl(f^n(P)\bigr)=0$ for
some~$n$. We also note that the arithmetic degree is independent of
the choice of ample height function~$h_X$; see
Proposition~\ref{proposition:afPindepofht}.
\end{remark}

\begin{remark}
Let~$f:X\to X$ be a morphism with~$\d_f>1$, and let~$D\in\Pic(X)_\RR$ be
an ample divisor class satisfying the linear equivalence
$f^*D\sim\d_fD$. Then using properties of the classical canonical
height~$\hhat_{D,f}$, as described for example
in~\cite{callsilv:htonvariety}, it is an exercise to show that 
\[
  \hhat_f(P)>0\Longrightarrow\a_f(P)=\d_f.
\]
In the number field case, it is also an exercise to prove that
\[
  \hhat_f(P)=0\Longrightarrow\#\Orbit_f(P)<\infty,
\]
so in particular, Conjecture~\ref{conjecture:afPeqdfY} is true in this
case.  There are other situations in which one can define a canonical
height having sufficiently good properties to prove
Conjecture~\ref{conjecture:afPeqdfY}; see
Section~\ref{section:caseafPeqdf}
and~\cite{kawsilv:dfeqafPexamples,arxiv1111.5664} for examples and
further details. But in general, a rational map, or even a morphism,
does not have a canonical height with sufficiently good properties to
directly imply Conjecture~\ref{conjecture:afPeqdfY}(d).  The arithmetic
degree~$\a_f(P)$, although coarser than an ample canonical height,
may be viewed as a general non-trivial measure of the arithmetic
complexity of the $f$-orbit of~$P$.
\end{remark}

\section{Basic Properties of the Arithmetic Degree}
\label{section:defs}

In this section we verify that the upper and lower arithmetic degrees
are well-defined, independent of the choice of height function~$h_X$
on~$X$, and we prove a counting result for points in orbits.
We also prove two useful lemmas.

\begin{proposition}
\label{proposition:afPindepofht}
The upper and lower arithemtic degrees~$\aupper_f(P)$
and $\alower_f(P)$ are independent of the choice of the height
function~$h_X$.
\end{proposition}
\begin{proof}
If~$P$ has finite~$f$-orbit, then it is clear from the definition that
the limit~$\a_f(P)$ exists and is equal to~$1$, regardless of the
choice of~$h_X$. We assume henceforth that~$P$ is not preperiodic,
which means that we can replace~$\hplus_X$ with~$h_X$ when taking
limits over the orbit of~$P$
\par
Let~$h$ and~$h'$ be heights on~$X$ relative to ample divisors~$D$
and~$D'$, and let the corresponding arithmetic degrees be denoted
respectively by $\aupper_f(P)$, $\alower_f(P)$, $\aupper'_f(P)$, and
$\alower'_f(P)$.  By definition of
ampleness~\cite[Section~II.7]{hartshorne}, there is an integer~$m$
such that~$mD-D'$ is ample, so standard functorial properties of
height
functions,
as described for example in~\cite{lang:diophantinegeometry}
of~\cite[Theorem~B.3.2]{hindrysilverman:diophantinegeometry},
imply that there is a non-negative constant~$C$ such that
\begin{equation}
  \label{eqn:mhQ}
  mh(Q)\ge h'(Q) - C\quad\text{for all $Q\in X(\Kbar)$.}
\end{equation}
We choose a sequence of indices~$\Ncal\subset\NN$ such that
\begin{equation}
  \label{eqn:limhprime}
  \lim_{n\in\Ncal} h'\bigl(f^n(P)\bigr)^{1/n}
  = \limsup_{n\to\infty} h'\bigl(f^n(P)\bigr)^{1/n}
  = \aupper'_f(P).
\end{equation}
Then
\begin{align*}
  \aupper'_f(P)
  &= \lim_{n\in\Ncal} h'\bigl(f^n(P)\bigr)^{1/n} 
    &&\text{from \eqref{eqn:limhprime},}\\
  &\le \lim_{n\in\Ncal} \bigl(mh\bigl(f^n(P)\bigr)+C\bigr)^{1/n} 
    &&\text{from \eqref{eqn:mhQ},}\\
  &\le \limsup_{n\to\infty} \bigl(mh\bigl(f^n(P)\bigr)+C\bigr)^{1/n} \\
  &= \limsup_{n\to\infty} h\bigl(f^n(P)\bigr)^{1/n}
  = \aupper_f(P).
\end{align*}
This gives one inequality for the upper arithmetic degrees, and
reversing the roles of~$h$ and~$h'$ gives the opposite inequality,
which proves that $\aupper_f'(P)=\aupper_f(P)$.  We omit the similar
proof that $\alower_f'(P)=\alower_f(P)$.
\end{proof}

The following lemma says that~$\aupper_f(P)$ and~$\alower_f(P)$ depend
only on the eventual orbit of~$P$.

\begin{lemma}
\label{lemma:affkPeqafP}
Let $f:X\dashrightarrow X$ be a rational map defined over~$\Kbar$.
Then for all $P\in X_f(\Kbar)$ and all~$k\ge0$,
\[
  \aupper_f\bigl(f^k(P)\bigr) = \aupper_f(P)
  \quad\text{and}\quad
  \alower_f\bigl(f^k(P)\bigr) = \alower_f(P).
\]
\end{lemma}
\begin{proof}
We compute
\begin{align*}
  \aupper_f\bigl(f^k(P)\bigr)
  &= \limsup_{n\to\infty} \hplus_X\bigl(f^{n+k}(P)\bigr)^{1/n}\\
  &= \limsup_{n\to\infty} 
          \left(\hplus_X\bigl(f^{n+k}(P)\bigr)^{1/(n+k)}\right)^{1+k/n}\\
  &= \limsup_{n\to\infty} \hplus_X\bigl(f^{n+k}(P)\bigr)^{1/(n+k)}\\
  &=\aupper_f(P).
\end{align*}
The proof for~$\alower_f$ is similar, which completes the proof of
Lemma~\ref{lemma:affkPeqafP}.
\end{proof}

We next prove Proposition~\ref{proposition:limBQhXQleB}, which we
recall says that if the limit defining~$\a_f(P)$ exists, then the
growth of the height counting function of the orbit of~$P$ is given 
by~\eqref{eqn:limBQhXQleB}.

\begin{proof}[Proof of Proposition~$\ref{proposition:limBQhXQleB}$]
Since~$\#\Orbit_f(P)=\infty$, it suffices to
prove~\eqref{eqn:limBQhXQleB} with~$\hplus_X$ in place of~$h_X$.  For
every~$\e>0$ there is an~$n_0(\e)$ so that
\[
  (1-\e)\a_f(P) \le \hplus_X\bigl(f^n(P)\bigr)^{1/n} \le (1+\e)\a_f(P)
  \quad\text{for all $n\ge n_0(\e)$.}
\]
It follows that
\[
  \bigl\{n\ge n_0(\e) : (1+\e)\a_f(P) \le B^{1/n}\bigr\} 
  \subset\bigl\{n\ge n_0(\e) : \hplus_X\bigl(f^n(P)\bigr) \le B\bigr\}
\]
and
\[
  \bigl\{n\ge n_0(\e) : \hplus_X\bigl(f^n(P)\bigr) \le B\bigr\}
  \subset \bigl\{n\ge n_0(\e) : (1-\e)\a_f(P) \le B^{1/n}\bigr\}.
\]
Counting the number of elements in these sets yields
\begin{align*}
  \frac{\log B}{\log\bigl((1+\e)\a_f(P)\bigr)} - 1
  &\le \# \bigl\{n\ge 0 : \hplus_X\bigl(f^n(P)\bigr) \le B\bigr\} \\
  \noalign{\noindent and}
   \# \bigl\{n\ge 0 : \hplus_X\bigl(f^n(P)\bigr) \le B\bigr\}
  &\le  \frac{\log B}{\log\bigl((1-\e)\a_f(P)\bigr)} + n_0(\e) + 1.
\end{align*}
Dividing by~$\log B$ and letting~$B\to\infty$ gives
\[
  \frac{1}{\log\bigl((1+\e)\a_f(P)\bigr)}
  \le \liminf_{B\to\infty} 
   \frac{\#\bigl\{Q\in\Orbit_f(P) : \hplus_X(Q)\le B\bigr\}}{\log B}
\]
and
\[
  \limsup_{B\to\infty} \frac{\#\bigl\{Q\in\Orbit_f(P) 
        : \hplus_X(Q)\le B\bigr\}}{\log B}
  \le \frac{1}{\log\bigl((1-\e)\a_f(P)\bigr)}.
\]
Since~$\e$ is arbitrary, and the liminf is less than or equal 
to the limsup, this completes the proof that
\[
  \lim_{B\to\infty} \frac{\#\bigl\{Q\in\Orbit_f(P)
            : \hplus_X(Q)\le B\bigr\}}{\log B}
  = \frac{1}{\log\a_f(P)},
\]
including the fact that if~$\a_f(P)=1$, then the limit is~$\infty$.
\end{proof}

The following elementary linear algebra result will be used in
the proof of Theorem~\ref{theorem:afPledfintro}.

\begin{lemma}
\label{lemma:rhoAAn1n}
Let~$A=(a_{ij})\in M_r(\CC)$ be an $r$-by-$r$ matrix.
Let~$\|A\|=\max|a_{ij}|$, and as usual let~$\rho(A)$ denote the
spectral radius of~$A$. Then there are constants~$c_1$ and~$c_2$,
depending on~$A$, such that
\begin{equation}
  \label{eqn:anLogAnlogrA}
    c_1 \rho(A)^n \le \|A^n\| \le c_2 n^r \rho(A)^n
  \quad\text{for all $n\ge0$.}
\end{equation}
In particular, we have $\rho(A) = \lim_{n\to\infty} \|A^n\|^{1/n}$.
\end{lemma}
\begin{proof}
For any matrices~$A$ and~$B$ in~$M_r(\CC)$,the triangle inequality
gives the estimate
\[
  \|AB\| \le r\|A\|\cdot\|B\|.
\]
We write~$A=P\L P^{-1}$ with~$\L$ in Jordan normal form.  Let~$\l$ be
an eigenvalue of~$A$ having largest absolute value such that among
such largest eigenvalues, it has the largest Jordan block. Let the
dimension of the largest $\l$-Jordan block be~$\ell$. Then
\[
  \|\L^n\| = \max_{0\le i<\ell} \left\{ \binom{n}{i}|\l|^{n-i} \right\}.
\]
Since $r\le \ell$ and $|\l|=\rho(A)$, the trivial
estimates $1\le\binom{n}{i}\le n^r$ gives
\begin{equation}
  \label{eqn:rAnr}
  \rho(A)^{n-r} \le \|\L^n\| \le n^r\rho(A)^n.
\end{equation}
We next observe that
\begin{align*}
  \|A^n\|&=\|P\L^nP^{-1}\| \le r^2\|P\|\cdot\|P^{-1}\|\cdot\|\L^n\|,\\
  \|\L^n\|&=\|P^{-1}A^nP\| \le r^2\|P^{-1}\|\cdot\|P\|\cdot\|A^n\|,
\end{align*}
so setting $C=C(A)=r^2\|P\|\cdot\|P^{-1}\|>0$, we have 
\begin{equation}
  \label{eqn:CLnAnCLn}
  C^{-1}\|\L^n\| \le \|A^n\| \le C\|\L^n\|
  \quad\text{for all $n\ge0$.}
\end{equation}
Combining~\eqref{eqn:rAnr} and~\eqref{eqn:CLnAnCLn}
gives~\eqref{eqn:anLogAnlogrA}, and then taking~$n^{\text{th}}$-roots
and letting~$n\to\infty$ gives $\|A^n\|^{1/n}\to\rho(A)$.
\end{proof}

\section{A divisor inequality for rational maps}
\label{section:divisorinequality}

Let $f:X\dashrightarrow X$ be a rational map. Our goal in this section
is to prove the following geometric inequality relating the actions
of~$(f^*)^n$ and~$(f^n)^*$ on the vector space~$\NS(X)_\RR$. This
result will provide a crucial estimate in our proof
that $h_X\circ f^n\ll(\d_f+\e)^nh_X$.

\begin{theorem}
\label{thm:AfmleCAfm}
Let $X$ be a smooth projective variety, and fix a basis $D_1, \ldots,
D_r$ for $\NS(X)_\RR$.  A dominant rational map $g: X
\dasharrow X$ induces a linear map on $\NS(X)_\RR$, and we write
\[
  g^* D_j \equiv \sum_{i=1}^r a_{ij}(g) D_i
  \quad\text{and}\quad A(g)=\bigl(a_{ij}(g)\bigr) \in M_r(\RR).
\]
We also let~$\|\,\cdot\,\|$ denote the sup norm on~$M_r(\RR)$.
Then for any dominant rational map $f:X\dashrightarrow X$
there is a constant $C=C(f,D_1,\ldots,D_r)>0$ such that
\begin{align}
  \label{eqn:AfmnleCAfmAfn}
  \bigl\|A(f^{m+n})\bigr\| 
      &\le C \bigl\|A(f^m)\bigr\| \cdot\bigl\|A(f^n)\bigr\|
            \quad\text{for all $m,n\ge1$,}  \\
  \label{eqn:AfmleCAfm}
  \bigl\|A(f^{m})\bigr\| 
      &\le C \bigl\|A(f)^m\bigr\|
            \quad\text{for all $m\ge1$.}  
\end{align}
\end{theorem}

We remark that an immediate corollary is the convergence of the limit
defining the dynamical degree.

\begin{corollary}
\label{corollary:dfexists}
The limit $\d_f=\lim_{n\to\infty}\rho\bigl((f^n)^*,\NS(X)_\RR\bigr)^{1/n}$
converges.
\end{corollary}
\begin{proof}
With notation as in the statement of Theorem~\ref{thm:AfmleCAfm},
we have $\rho\bigl((f^n)^*,\NS(X)_\RR\bigr)=\rho\bigl(A(f^n)\bigr)$,
so~\eqref{eqn:AfmnleCAfmAfn} gives
\[
  \log\rho\bigl((f^{m+n})^*\bigr)
  \le \log\rho\bigl((f^{m})^*\bigr) + \log\rho\bigl((f^{n})^*\bigr)
  + O(1).
\]
Using this convexity estimate, it is an exercise to show that the
sequence~$\frac{1}{n}\rho\bigl((f^n)^*\bigr)$ is Cauchy.
\end{proof}

We start the proof of Theorem~\ref{thm:AfmleCAfm} with a preliminary
result relating~$(g\circ f)^*$ and~$f^*\circ g^*$.  This is
essentially shown in \cite[Proof of Proposition~1.2(ii)]{MR2179389} by
an analytic argument; cf.\ the equation labeled~($\dag$)
in~\cite{MR2179389}. We give an algebraic proof.

\begin{proposition}
\label{proposition:fgDminusgfD}
Let $X, Y, Z$ be smooth projective varieties of the same dimension,
and let $f: X \dasharrow Y$ and $g: Y \dasharrow Z$ be dominant
rational maps. Let $D$ be a nef divisor on $Z$. 
Then for any nef divisor $H$ on $X$, we have
\begin{equation}
  \label{eqn:fgstardHn}
  (g \circ f)^* D \cdot H^{N-1}
  \leq  f^*( g^*D) \cdot H^{N-1}. 
\end{equation}
\end{proposition}
\begin{proof}
We blow up the indeterminacy locus $I_f$ of $f$ so that we have a
smooth projective variety $\widetilde{X}$, a birational morphism
$\pi_{\widetilde{X}}: \widetilde{X} \to X$, and a morphism
$\widetilde{f}: \widetilde{X} \to Y$ such that $\widetilde{f} = f
\circ \pi_{\widetilde{X}}$ (\cite[Example~II.7.17.3]{hartshorne}).
Similarly, we blow up the indeterminacy
locus $I_g$ of $g$ so that we have a smooth projective variety
$\widetilde{Y}$, a birational morphism $\pi_{\widetilde{Y}}:
\widetilde{Y} \to Y$, and a morphism $\widetilde{g}: \widetilde{Y} \to
Z$ such that $\widetilde{g} = g \circ \pi_{\widetilde{Y}}$.
\par
Let $h: \widetilde{X} \dasharrow \widetilde{Y}$ be the induced
dominant rational map. We blow up the indeterminacy locus $I_h$ of $h$
so that we have a smooth projective variety $W$, a birational morphism
$\pi_{W}: W \to \widetilde{X}$, and a morphism $\widetilde{h}: W \to
\widetilde{Y}$ such that $\widetilde{h} = h \circ \pi_{W}$. The varieties
and maps are illustrated in Figure~\ref{figure:blowup}
\par
Since nef divisors are limits of ample divisors, we may assume
that~$D$ is ample. Replacing $D$ by $kD$ for sufficiently large $k$,
we may assume that $D$ is very ample and represented by an effective
divisor with the following properties:
\begin{align}
  &\parbox[t]{.85\hsize}
      {\hangindent1em\hangafter1\noindent
        \hbox to1em{\textbullet\hfil}$D$ does not contain the image of any
        divisor in~$W$ that maps to a smaller dimensional variety
        in~$Y$ or~$Z$. Also, $D$ does not contain the image of any
        divisor in~$W$ whose image in~$X$ is contained in the
        Zariski closue of $(f|_{X\setminus I_f})^{-1}(I_g)$.}
    \label{eqn:image1} \\
  &\parbox[t]{.85\hsize}
      {\hangindent1em\hangafter1\noindent
        \hbox to1em{\textbullet\hfil}$D$ does not contain the image of any
        divisor in~$\widetilde Y$ that maps to a smaller dimensional variety
        in~$Y$ or~$Z$.}
    \label{eqn:image2} 
\end{align}
With these assumptions, we claim that the divisor
\begin{equation}
  \label{eqn:fgstard}
  f^*(g^*D) - (g \circ f)^* D
\end{equation}
is effective, and hence has non-negative intersection with~$H^{N-1}$.

\begin{figure}[ht]
\[
\begin{array}{ccccc}
  W 
  \\[10pt]
  {\scriptstyle\pi_W}\Big\downarrow\phantom{\scriptstyle\pi_W}
    &\smash{\raisebox{5pt}{$\stackrel{\;\widetilde{h}}{\searrow}$}} 
  \\[10pt]
  \widetilde{X} 
    &\stackrel{h}{\dashrightarrow}
    &\widetilde{Y} 
  \\[10pt]
  {\scriptstyle\pi_{\widetilde X}}\Big\downarrow
              \phantom{\scriptstyle\pi_{\widetilde X}}
    &\smash{\raisebox{0pt}{$\stackrel{\;\widetilde{f}}{\searrow}$}} 
    &{\scriptstyle\pi_{\widetilde Y}}\Big\downarrow
              \phantom{\scriptstyle\pi_{\widetilde Y}}
    &\smash{\raisebox{0pt}{$\stackrel{\;\widetilde{g}}{\searrow}$}} 
  \\[10pt]
  X &\stackrel{f}{\dashrightarrow} & Y &\stackrel{g}{\dashrightarrow} & Z \\
\end{array}
\]
\caption{Resolution of the maps~$f$,~$g$, and~$g\circ f$}
\label{figure:blowup}
\end{figure} 

We note that~\eqref{eqn:image2} implies that
\[
  g^*D = \text{Zariski closure of $\left(g|_{Y\setminus I_g}\right)^*D$ in~$Y$,} 
\]
and similarly~\eqref{eqn:image1} implies that
\[
  (g\circ f)^*D 
   = \text{Zariski closure of 
    $\left(g\circ f|_{X \setminus (I_f \cup (f_{X\setminus I_f})^{-1}(I_g))}\right)^*D$
        in $Y$.}
\]
The divisors $(g\circ f)^*D$ and $f^*(g^*D)$ agree on 
$X \setminus (I_f \cup (f_{X\setminus I_f})^{-1}(I_g))$, so we see that
$f^*(g^*D)-(g\circ f)^*D$ is effective, which
completes the proof of Proposition~\ref{proposition:fgDminusgfD}.
\end{proof}

We now give the proof of Theorem~\ref{thm:AfmleCAfm}.

\begin{proof}[Proof of Theorem~$\ref{thm:AfmleCAfm}$]
We set the following notation.
\begin{notation}
\item[$N$]
the dimension of $X$, which we assume is at least~$2$.
\item[$\Amp(X)$]
the ample cone in $\NS(X)_\RR$ of all ample $\RR$-divisors.
\item[$\Nef(X)$]
the nef cone in $\NS(X)_\RR$ of all nef $\RR$-divisors.
\item[$\Eff(X)$]
the effective cone in $\NS(X)_\RR$ of all effective $\RR$-divisors.
\item[$\overline{\Eff}(X)$]
the pseudoeffective cone, i.e., the $\RR$-closure of~$\Eff(X)$.
\end{notation}
As explained in~\cite[Section~1.4]{MR1644323}, we have
\[
  \Nef(X)=\overline{\Amp(X)}
  \quad\text{and}\quad
  \Amp(X)=\interior\bigl(\Nef(X)\bigr).
\]
In particular,~$\Nef(X)$ is a closed convex cone.
Also, since~$\Amp(X)\subset\Eff(X)$, it follows that
$\Nef(X)\subset\overline{\Eff}(X)$.

\begin{lemma}
\label{lemma:DEffXHAmp} 
With notation as above,
let $D\in\overline{\Eff}(X)\setminus\{0\}$ and $H\in\Amp(X)$. 
Then $D\cdot H^{N-1}>0$.
\end{lemma}
\begin{proof}
Since~$H$ is ample and~$D$ is in the closure of the effective cone,
we certainly have~$D\cdot H^{N-1}\ge0$. Our goal is to prove that
we have a strict inequality.
\par
We first consider the case $N=2$. Since $D\ne0$ in~$\NS(X)_\RR$, there
is a divisor~$E$ such that~$D\cdot E\ne0$. Replacing~$E$ by~$-E$ if
necessary, we may assume that $D\cdot E<0$. Choose~$k>0$ sufficiently
large so that $kH+E$ is ample. Since~$D$ is a limit of effective
divisors, we have $D\cdot(kH+E)\ge0$. Hence
\[
  D\cdot H = \frac{-D\cdot E}{k} > 0.
\]
\par
We now proceed by induction on~$N$. Let $N=\dim X\ge3$. Replacing~$H$
with~$kH$ for an appropriate~$k\ge1$, we may assume that~$H$
is very ample. Let~$Y$ be a (smooth) irreducible variety
in the linear system~$|H|$. The Lefschetz hyperplane 
theorem~\cite[Theorem~1.23]{MR1997577}  says that the restriction map
$\NS(X)\to\NS(Y)$ is injective and preserves effective divisors. 
Our induction hypothesis says that
\[
  D|_Y \cdot (H|_Y)^{N-2} > 0.
\]
Hence $D\cdot Y\cdot H^{N-2}>0$. But $Y\sim H$ in~$\Pic(X)$, so in
particular $Y\equiv H$ in~$\NS(X)_\RR$. Hence
\text{$D\cdot{H^{N-1}}>0$}, which completes the proof of
Lemma~\ref{lemma:DEffXHAmp}.
\end{proof}

\begin{lemma}
\label{lemma:absvleCvHN}
Let~$H\in\Amp(X)$, and fix some norm~$|\,\cdot\,|$ on the $\RR$-vector
space~$\NS(X)_\RR$. There are constants~$C_1,C_2>0$ such that
\begin{equation}
  \label{eqn:vleCvHN1}
  C_1 |v| \le  v\cdot H^{N-1} \le C_2 |v|
  \quad\text{for all $v\in\overline{\Eff}(X)$.}
\end{equation}
In particular, the inequality~\eqref{eqn:vleCvHN1} holds for
all $v\in\Nef(X)$.
\end{lemma}
\begin{proof}
We consider the map
\[
  \f:\NS(X)_\RR\longrightarrow\RR,\quad
  \f(w) = w\cdot H^{N-1}.
\]
Since~$\f$ is continuous, it attains a minimum and (finite) maximum
when restricted to the compact set
\[
  \overline{\Eff}(X)\cap \bigl\{w\in\NS(X)_\RR : |w|=1\bigr\}.
\]
Lemma~\ref{lemma:DEffXHAmp} tells us that~$\f(w)>0$
for all nonzero~$w\in\overline{\Eff}(X)$, so the minimum is strictly
positive, say
\[
  C_1 = \inf \bigl\{\f(w) : w\in\overline{\Eff}(X)~\text{and}~|w|=1\bigr\} > 0.
\]
Then for all $v\in\overline{\Eff}(X)\setminus\{0\}$ we have
\[
  v\cdot H^{n-1}
  =  |v| \f\left( \frac{v}{|v|} \right)
  \ge C_1 |v|.
\]
Similarly, letting
\[
  C_2 = \sup \bigl\{\f(w) : w\in\overline{\Eff}(X)~\text{and}~|w|=1\bigr\} 
  < \infty,
\]
we have
\[
  v\cdot H^{n-1}
  =  |v| \f\left( \frac{v}{|v|} \right)
  \le C_2 |v|.
\]
This proves the first part of Lemma~\ref{lemma:absvleCvHN}, and the
last assertion is then clear, since as noted earlier,
$\Nef(X)\subseteq\overline{\Eff}(X)$.
\end{proof}

We resume the proof of Theorem~\ref{thm:AfmleCAfm}.  As in the
proof of Lemma~\ref{lemma:absvleCvHN}, we fix a norm~$|\,\cdot\,|$ on
the $\RR$-vector space~$\NS(X)_\RR$, and for any linear map
$A:\NS(X)_\RR\to\NS(X)_\RR$, we set
\[
  \|A\|' = \sup_{v\in\Nef(X)\setminus0} \frac{|Av|}{|v|}.
\]
We note that for linear maps $A,B\in\End\bigl(\NS(X)_\RR\bigr)$ and
$c\in\RR$ we have
\[
  \|A+B\|'\le\|A\|'+\|B\|'
  \quad\text{and}\quad
  \|cA\|' = |c|\,\|A\|'.
\]
Further, since~$\Nef(X)$ generates~$\NS(X)_\RR$ as an $\RR$-vector
space, we have~$\|A\|'=0$ if and only if $A=0$. Thus~$\|\,\cdot\,\|'$
is an $\RR$-norm on~$\End\bigl(\NS(X)_\RR\bigr)$.
\par
Similarly, for any linear map $A:\NS(X)_\RR\to\NS(X)_\RR$, we set
\[
  \|A\|'' = \sup_{w\in\overline{\Eff}(X)\setminus0} \frac{|Aw|}{|w|},
\]
then $\|\,\cdot\,\|''$ is an $\RR$-norm
on~$\End\bigl(\NS(X)_\RR\bigr)$. 
\par
The maps~$(f^m)^*$ for $m\ge1$ preserve $\overline{\Eff}(X)$. 
This allows us to compute
\begin{align*}
  \bigl\|&(f^{m+n})^*\bigr\|' \\
  &= \sup_{v\in\Nef(X)\setminus0} \frac{\bigl|(f^{m+n})^*v\bigr|}{|v|} \\
  &\le C_1^{-1} \sup_{v\in\Nef(X)\setminus0} \frac{(f^{m+n})^*v\cdot H^{N-1}}{|v|} 
    &&\text{from Lemma \ref{lemma:absvleCvHN},} \\
  &\le C_1^{-1} \sup_{v\in\Nef(X)\setminus0} 
         \frac{(f^m)^*((f^{n})^*v)\cdot H^{N-1}}{|v|} 
    &&\text{from Proposition~\ref{proposition:fgDminusgfD},} \\
  &=  C_1^{-1} \sup_{v\in\Nef(X)\setminus0} \left(
         \frac{(f^m)^*((f^{n})^*v)\cdot H^{N-1}}{|(f^n)^*v|} 
         \cdot \frac{|(f^n)^*v|}{|v|} \right) \hidewidth\\
  &\le  C_1^{-1} \left( \sup_{v\in\Nef(X)\setminus0} 
         \frac{(f^m)^*((f^{n})^*v)\cdot H^{N-1}}{|(f^n)^*v|} \right)
         \cdot \left( \sup_{v\in\Nef(X)\setminus0} \frac{|(f^n)^*v|}{|v|} \right) 
         \hidewidth\\
  &=  C_1^{-1} \left( \sup_{v\in\Nef(X)\setminus0} 
         \frac{(f^m)^*((f^{n})^*v)\cdot H^{N-1}}{|(f^n)^*v|} \right)
          \cdot \bigl\|(f^n)^*\bigr\|' 
          \hidewidth\\
  &\le 
     \smash[b]{
     C_1^{-1} \left( \sup_{w\in\overline{\Eff}(X)\setminus0} 
         \frac{(f^m)^*w\cdot H^{N-1}}{|w|} \right)
          \cdot \bigl\|(f^n)^*\bigr\|' 
     }
     \hidewidth \\
  &&&\text{since $\Nef(X)\subseteq\overline{\Eff}(X)$,}  \\
  &\le C_1^{-1}C_2 \left( \sup_{w\in\overline{\Eff}(X)\setminus0} 
         \frac{|(f^m)^*w|}{|w|} \right)
          \cdot \bigl\|(f^n)^*\bigr\|'
    &&\text{from Lemma \ref{lemma:absvleCvHN},} \\
  &\le C_1^{-1}C_2 \bigl\|(f^m)^*\bigr\|'' \cdot \bigl\|(f^n)^*\bigr\|'.
\end{align*}
We recall that we have defined~$\|\,\cdot\,\|$ to be the sup norm
on~$M_r(\RR)=\End\bigl(\NS(X)_\RR\bigr)$, where the identification is
via the given basis~$D_1,\ldots,D_r$ of~$\NS(X)_\RR$. We thus have
three norms~$\|\,\cdot\,\|$,~$\|\,\cdot\,\|'$, and~$\|\,\cdot\,\|''$
on $\End\bigl(\NS(X)_\RR\bigr)$, so there are positive
constants~$C_3'$,~$C_4'$,~$C_3''$ and~$C_4''$ such that
\[
  C'_3\|\g\| \le \|\g\|' \le C'_4\|\g\|
  \quad\text{and}\quad
  C''_3\|\g\| \le \|\g\|'' \le C''_4\|\g\|
\]
for all $\g\in\End\bigl(\NS(X)_\RR\bigr)$. Hence
\begin{align*}
  \bigl\|A(f^{n+m})\bigr\|
  = \bigl\|(f^{n+m})^*\bigr\|
  &\le {C_3'}^{-1} \bigl\|(f^{n+m})^*\bigr\|' \\
  &\le {C_3'}^{-1}C_1^{-1}C_2 \|(f^n)^*\|'\cdot \bigl\|(f^m)^*\bigr\|'' \\
  &\le {C_3'}^{-1}C_1^{-1}C_2C_4'C_4'' \|(f^n)^*\|\cdot \bigl\|(f^m)^*\bigr\| \\
  &=  {C_3'}^{-1}C_1^{-1}C_2C_4'C_4'' \|A(f^n)\|\cdot \bigl\|A(f^m)\bigr\|.
\end{align*}
This completes the proof of~\eqref{eqn:AfmnleCAfmAfn}.
\par
A similar calculation gives
\begin{align*}
  \bigl\|(f^m)^*\bigr\|'
  &= \sup_{v\in\Nef(X)\setminus0} \frac{\bigl|(f^m)^*v\bigr|}{|v|} \\
  &\le C_1^{-1} \sup_{v\in\Nef(X)\setminus0} \frac{(f^m)^*v\cdot H^{N-1}}{|v|} 
    &&\text{from Lemma \ref{lemma:absvleCvHN},} \\
  &\le C_1^{-1} \sup_{v\in\Nef(X)\setminus0} \frac{(f^*)^m v\cdot H^{N-1}}{|v|} 
    &&\text{from Proposition~\ref{proposition:fgDminusgfD},} \\
  &\le C_1^{-1}C_2 \sup_{v\in\Nef(X)\setminus0} \frac{\bigl|(f^*)^m v\bigr|}{|v|} 
    &&\text{from Lemma \ref{lemma:absvleCvHN},} \\
  &=  C_1^{-1}C_2 \bigl\|(f^*)^m\bigr\|'.
\end{align*}
Hence
\begin{align*}
  \bigl\|A(f^m)\bigr\|
  &= \bigl\|(f^m)^*\bigr\|
  \le {C_3'}^{-1} \bigl\|(f^m)^*\bigr\|'
  \le {C_3'}^{-1}C_1^{-1}C_2 \bigl\|(f^*)^m\bigr\|' \\
  &\le {C_3'}^{-1}C_1^{-1}C_2C_4' \bigl\|(f^*)^m\bigr\|
  = {C_3'}^{-1}C_1^{-1}C_2C_4' \bigl\|A(f)^m\|.
\end{align*}
This completes the proof of~\eqref{eqn:AfmleCAfm},
and with it the proof of Theorem~\ref{thm:AfmleCAfm}.
\end{proof}

\begin{remark}
If we assume that $f:X\to X$ is a morphism, then the conclusions of
Theorem~\ref{thm:AfmleCAfm} are valid for normal varieties~$X$.
Indeed, in this situation it suffices to work with~$\Nef(X)$; there
is no need to introduce $\overline{\Eff}(X)$ into the argument.
\end{remark}

\section{A height inequality for rational maps}
\label{section:htinequality}

Let $f:X\dashrightarrow X$ be a rational map and $D$ a divisor
on~$X$. Our goal in this section is to prove an arithmetic inequality
relating the height functions~$h_D\circ f$ and $h_{f^*D}$.  For
rational self-maps $f:\PP^N\dashrightarrow\PP^N$ of projective space,
the desired result follows by an elementary triangle inequality
argument~\cite[Theorem~B.2.5(a)]{hindrysilverman:diophantinegeometry},
but the proof for general varieties $f:X\dashrightarrow X$ is more
complicated because the pullback of an ample divisor by~$f$ need not
be ample.  With an eye towards future applications, and since the
argument is no more difficult, we prove a stronger result in which the
domain and range may be different varieties.  We again refer the
reader to~\cite{bombierigubler, hindrysilverman:diophantinegeometry,
  lang:diophantinegeometry, MR2316407, MR2514094} for the theory of
height functions and Weil's height machine.  In
Section~\ref{section:htinequalityaltpf} we will give an alternative
proof of Proposition~\ref{proposition:hDflehfD} that avoids blowups.

\begin{proposition}
\label{proposition:hDflehfD}
Let $X/\Kbar$ and $Y/\Kbar$ be smooth projective varieties, let
$f:Y\dashrightarrow X$ be a dominant rational map defined
over~$\Kbar$, let $D\in\Div(X)$ be an ample divisor, and
fix Weil height functions~$h_{X,D}$ and $h_{Y,f^*D}$ associated
to~$D$ and~$f^*D$.  Then
\[
  h_{X,D}\circ f(P) \le h_{Y,f^*D}(P) + O(1)
  \quad\text{for all~$P\in (Y\setminus I_f)(\Kbar)$,}
\]
where the~$O(1)$ bound depends on~$X$,~$Y$,~$f$, and the choice of height
functions, but is independent of~$P$.
\end{proposition}

\begin{proof}
We blow up the indeterminacy locus~$I_f$ of~$f$ to get a smooth
projective variety~$Z$, a birational morphism~$p: Z \to Y$, and a
morphism~$g: Z \to X$ such that~$f = g \circ p^{-1}$. For any
effective divisor~$D$ on~$X$, the pullback~$f^*D$ is defined by
\[
  f^*D = p_* (g^* D). 
\]
We note that~$f^*D$ is independent of the choice of~$Z$. 

\begin{lemma}
\label{lemma:pgDgDeff}
With notation as above, assume that~$D$ is nef. Then the divisor~$p^*
p_*(g^* D) - g^* D$ is effective.
\end{lemma}

\begin{proof}
We set~$B= p^* p_*(g^* D) - g^* D$. For any curve~$C$ on~$Z$ such
that~$p(C)$ is a point, we have
\[
  - B \cdot C = (g^* D) \cdot C - (p^* p_*(g^* D))\cdot C = 
   (g^* D) \cdot C \geq 0
\]
Thus~$-B$ is~$p$-nef. It follows from the negativity lemma (see
\cite[Lem\-ma~3.39]{MR1658959}) that~$B$ is effective if and only
if~$p_* B$ is effective. Since~$p_* B = 0$, we conclude that~$B$ is
effective.
\end{proof}

We now resume the proof of Proposition~\ref{proposition:hDflehfD}, so
in particular we assume that~$D$ is ample.  For a sufficiently large
$m$, the divisor~$mD$ is very ample, so there exists an effective
divisor~$D'$ that is linearly equivalent to~$mD$. Since
$f^*D'$ is linearly equivalent to~$f^*(m D)$, we may assume that
$D$ is effective.

We set 
\[
  B= p^* p_*(g^* D) - g^* D.
\]
Lemma~\ref{lemma:pgDgDeff} tells us that~$B$ is an effective divisor
with the property that
\[
  p\bigl(\Supp(B)\bigr) \subset I_f.
\]
For any~$\tilde{P} \in Z(\Kbar) \setminus \Supp(B)$, we estimate
$h_{p^* p_*(g^* D)} (\tilde{P})$ in two ways. 
First we have 
\begin{align}
\label{eqn:first}
  h_{p^* p_*(g^* D)} (\tilde{P}) 
  & = h_{g^* D + B} (\tilde{P}) \notag \\
  & = h_{g^* D}(\tilde{P}) + h_B(\tilde{P}) + O(1) \notag\\
  & \geq h_{g^* D}(\tilde{P}) +O(1), 
\end{align}
where the last inequality follows from the positivity of the height
$h_B$ on~$Z \setminus \Supp(B)$ for the effective divisor~$B$;
see~\cite[Theorem~B.3.2(e)]{hindrysilverman:diophantinegeometry}.
Secondly, we have
\begin{align}
\label{eqn:second}
  h_{p^* p_*(g^* D)} (\tilde{P}) 
  & = h_{p_*(g^* D)}(p(\tilde{P})) + O(1) \notag\\
  & = h_{f^* D}(p(\tilde{P})) +O(1). 
\end{align}
\par
Now let~$P\in Y(\Kbar) \setminus I_f$. Then there exists a unique
point~$\tilde{P} \in Z \setminus p^{-1}(I_f)$ with~$p(\tilde{P}) =
P$. Since~$\Supp(B) \subseteq p^{-1}(I_f)$, we have~$P \in Z \setminus
\Supp(B)$.  Hence
\begin{align*}
  h_{f^* D}(P) 
  &= h_{f^* D}(p(\tilde P)) 
    &&\text{since~$P=p(\tilde P)$,} \\
  & = h_{p^* p_*(g^* D)} (\tilde{P}) + O(1)
    &&\text{from \eqref{eqn:second},}\\
  & \geq h_{g^* D}(\tilde{P}) +O(1) 
    &&\text{from \eqref{eqn:first},}\\
  & = h_{D}(g(\tilde{P})) +O(1) 
    &&\text{since~$g$ is a morphism,} \\
 & = h_{D}(f(P)) + O(1)
    &&\text{since~$g(\tilde{P})=f(P)$.}
\end{align*}
This completes the proof of Proposition~\ref{proposition:hDflehfD}.
\end{proof} 

\begin{remark}
Proposition~\ref{proposition:hDflehfD} is true more generally for a
nef divisor~$D$ such that there exists an~$m \geq 1$ such that~$m D$
is linearly equivalent to an effective divisor.
\end{remark}

\section{A bound for the height of an iterate}
\label{section:afpledf}

We now prove the quantitative height upper bound for
$\hplus_X\bigl(f^n(P)\bigr)$ that constitutes one of the main results
of this paper. For the convenience of the reader, the statement
includes a reminder of the notation that we set in the introduction.

\begin{theorem}
\label{theorem:hXfnledfnhX}
\textup{(Theorem~\ref{theorem:hXfnlldfenhX})}
Let~$K$ be a global field, let $f:X\dashrightarrow X$ be a dominant
rational map defined over~$K$, let~$h_X$ be a Weil height
on~$X(\Kbar)$ relative to an ample divisor,
let~$\hplus_X=\max\{h_X,1\}$, and let~$\e>0$.  Then there is a
constant $C=C(X,h_X,f,\e)$ such that for all $P\in X_f(\Kbar)$ and all
$n\ge0$,
\[
  \hplus_X\bigl(f^n(P)\bigr) \le C\cdot(\d_f+\e)^n\cdot\hplus_X(P).
\]
\end{theorem}

Before proving Theorem~\ref{theorem:hXfnledfnhX}, we pause to show how
it immediately implies the fundamental inequality $\aupper_f(P)\le
\d_f$ stated in the introduction.

\begin{corollary}
\label{corollary:aupperfPledf}
\textup{(Theorem~\ref{theorem:afPledfintro})}
Let~$P\in X_f(\Kbar)$. Then
\begin{equation}
  \label{eqn:afPledf3} 
  \aupper_f(P)\le \d_f.
\end{equation}
\end{corollary}
\begin{proof}
Let $\e>0$. Then
\begin{align*}
  \aupper_f(P)
  &= \limsup_{n\to\infty} \hplus_X\bigl(f^n(P)\bigr)^{1/n}
    \quad\text{definition of $\aupper_f(P)$,} \\*
  &\le \limsup_{n\to\infty} \bigl( C\cdot (\d_f+\e)^n\cdot \hplus_X(P) \bigr)^{1/n}
     \quad\text{from Theorem~\ref{theorem:hXfnledfnhX},} \\*
  &= \d_f+\e.
\end{align*}
This holds for all $\e>0$, which proves that $\aupper_f(P)\le \d_f$.
\end{proof}

\begin{proof}[Proof of Theorem~$\ref{theorem:hXfnledfnhX}$]
If~$P$ is preperiodic, then~$\aupper_f(P)=1\le\d_f$, so there is
nothing to prove. We assume henceforth that~$\#\Orbit_f(P)=\infty$.
We let~$m$ and~$\ell$ be positive integers to be chosen later, and we
set
\[
  g = f^{m\ell}.
\]
We note that $X_f(\Kbar)\subset X_g(\Kbar)$.  We choose ample
divisors~$D_1,\ldots,D_r\in\Div(X)$ whose algebraic equivalence
classes form a basis for~$\NS(X)_\QQ$, and we fix height
functions~$h_{D_1},\ldots,h_{D_r}$ associated to the
divisors $D_1,\ldots,D_r$.  We note that any two ample heights are
commensurate with one another, i.e., $h_X\asymp h_X'$, so we may
take~$h_X$ to be
\[
   h_X(Q) = \max_{1\le i\le r} h_{D_i}(Q).
\]
To ease notation, we further assume that~$h_{D_1}$ is chosen to
satisfy~$h_{D_1}\ge1$, so $\hplus_X=h_X$.
\par
Applying~$g^*$ to the divisors in our basis of~$\NS(X)_\QQ$, we have
algebraic equivalences
\begin{equation}
  \label{eqn:gDkeqsumj1raikDi}
  g^*D_k \equiv \sum_{i=1}^r a_{ik}(g)D_i
  \quad\text{for some $a_{ik}(g)\in\QQ$.}
\end{equation}
We set the notation
\[
  A(g) = \bigl(a_{ik}(g)\bigr)
  \quad\text{and}\quad
  \bigl\|A(g)\bigr\| = \max_{i,k} \bigl|a_{ik}(g)\bigr|.
\]
Algebraic equivalences of divisors as in~\eqref{eqn:gDkeqsumj1raikDi}
implies a height relation as in the following result.

\begin{lemma}
\label{lemma:algeqht}
Let $E\in\Div(X)_\RR$ be a divisor that is algebraically equivalent
to~$0$, and fix a height function $h_E$ associated to~$E$. Then there
is a constant~$C=C(h_X,h_E)$ such that
\begin{equation}
  \label{eqn:hEPleCDEhDP}
  \bigl|h_E(P)\bigr| \le C\sqrt{\hplus_X(P)}
  \quad\text{for all $P\in X(\Kbar)$.}
\end{equation}
\end{lemma}
\begin{proof}
See for
example~\cite[Theorem~B.5.9]{hindrysilverman:diophantinegeometry}.
\end{proof}

\begin{remark}
A well-known weaker form of Lemma~\ref{lemma:algeqht} says that
\begin{equation}
  \label{eqn:hEPoverhDP}
  \lim_{\substack{P\in X(\Kbar)\\ h_X(P)\to\infty\\}}
  \frac{h_E(P)}{h_X(P)} = 0;
\end{equation}
see for example
\cite[Theorem~B.3.2(f)]{hindrysilverman:diophantinegeometry} or
\cite[Chapter~4, Proposition~3.3]{lang:diophantinegeometry}.  We
remark that it is possible to prove that \text{$\aupper_f(P)\le\d_f$} using
only the weaker estimate~\eqref{eqn:hEPoverhDP}, but in order to prove
the quantitative bound in Theorem~\ref{theorem:hXfnledfnhX} and the
error estimate in Theorem~\ref{theorem:hDfPbnhDfnP}, we need the
stronger estimate provided by~\eqref{eqn:hEPleCDEhDP}.
\end{remark}

Applying Lemma~\ref{lemma:algeqht} to~\eqref{eqn:gDkeqsumj1raikDi}
and using additivity of height functions,
we find a constant $C_1=C_1(\e,g)$ such that
\begin{equation}
  \label{multline:hgDQsumaikhDiQ}
  \left|h_{g^*D_k}(Q) - \sum_{i=1}^r a_{ik}(g)h_{D_i}(Q)\right| \le C_1\sqrt{h_X(Q)}
  \quad\text{for all $Q\in X(\Kbar)$.}
\end{equation}
Here and in what follows, the constants~$C_1,C_2,\ldots$ are allowed
to depend on the divisors~$D_1,\ldots,D_r$ and their associated height
functions, as well as on~$X$,~$f$,~$\e$,~$m$,~$\ell$, and~$\epsilon$.
However, we will eventually fix~$m$ and~$\ell$, at which point
\[
  C_i = C_i(X,f,\e,h_{D_1},\ldots,h_{D_r}).
\]
We also remind the reader that we have chosen~$h_X$ to
satisfy~$h_X\ge1$.
\par
We apply Proposition~\ref{proposition:hDflehfD} to the rational
map~$g$ and to each of the ample divisors~$D_1,\ldots,D_r$. Thus
for all points~$Q\in X(\Kbar)$, we have
\begin{align}
\label{eqn:hDkgQlerdAhXQ}
  h_X\bigl(g(Q)\bigr)
  &=\max_{1\le k\le r} h_{D_k}\bigl(g(Q)\bigr) 
    \quad\text{definition of $h_X$,} \notag\\ 
  &\le \max_{1\le k\le r} \bigl(h_{g^*D_k}(Q)+C_2\bigr)
  \quad\text{from Proposition~\ref{proposition:hDflehfD},} \notag\\
  &\le \max_{1\le k\le r} \left(\sum_{i=1}^r a_{ik}(g)h_{D_i}(Q)\right) 
        + C_3\sqrt{h_X(Q)}
   \quad\text{from \eqref{multline:hgDQsumaikhDiQ},} \notag\\
  &\le \left(r \max_{1\le i,k\le r} \bigl|a_{ik}(g)\bigr| \right)h_X(Q)
        +  C_3\sqrt{h_X(Q)} \notag\\
  &= r\bigl\|A(g)\bigr\|h_X(Q) + C_3\sqrt{h_X(Q)}.
\end{align}
\par 
We are going to use the following elementary lemma.

\begin{lemma}
\label{lemma:gShS0inf}
Let~$S$ be a set, let $g:S\to S$ and $h:S\to[0,\infty)$ be maps, let
$a\ge1$ and $b\ge1$ be constants. Suppose that for all $x\in S$ we have
\begin{equation}
  \label{eqn:hgzleahxcsq}
  h\bigl(g(x)\bigr) \le a h(x) + c\sqrt{h(x)}.
\end{equation}
Then for all $x\in S$ and all $n\ge0$, 
\begin{equation}
  \label{eqn:hgnxlean}
  h\bigl(g^n(x)\bigr)
  \le a^n \left(  h(x) + \bigl(2\sqrt2 c\bigr)^n\sqrt{h(x)} \right).
\end{equation}
\end{lemma}
\begin{proof}
The proof is an elementary induction on~$n$. For the convenience of
the reader, we give the details in Appendix~\ref{appendix:hgxlemma}.
\end{proof}

We apply Lemma~\ref{lemma:gShS0inf} to~\eqref{eqn:hDkgQlerdAhXQ} to
obtain
\begin{align}
  \label{eqn:hXgnQlerAg}
  h_X\bigl(g^n(Q)\bigr) 
  & \le \left(r\bigl\|A(g)\bigr\|\right)^n 
        \Bigl(h_X(Q) + C_4^n\sqrt{h_X(Q)}\Bigr) \notag\\
  & \le \left(C_5r\bigl\|A(g)\bigr\|\right)^n h_X(Q),
\end{align}
where we stress that~$C_4$ and $C_5$ do not depend on~$Q$ or~$n$.
\par
We recall that $g=f^{m\ell}$, which lets us estimate
\begin{align*}
  \bigl\|A(g)\bigr\|
  &= \bigl\|A\bigl((f^\ell)^m\bigr)\bigr\| \\
  &\le C_6 \bigl\|A(f^\ell)^m\bigr\|
    &&\text{Theorem~\ref{thm:AfmleCAfm} applied to $f^\ell$,} \\
  &\le C_7 m^r \rho\bigl(A(f^\ell)\bigr)^m
    &&\text{from Lemma~\ref{lemma:rhoAAn1n},}
\end{align*}
\par
By definition, the dynamical degree is the limit of
$\rho\bigl(A(f^\ell)\bigr)^{1/\ell}$ as~$\ell\to\infty$. So we now fix
an $\ell=\ell(\e,f)$ such that
\[
  \rho\bigl(A(f^\ell)\bigr) \le (\d_f+\e)^\ell.
\]
For this choice of~$\ell$, we have
\begin{equation}
  \label{eqn:AgleC4}
  \bigl\|A(g)\bigr\|
  \le C_7 m^r (\d_f+\e)^{\ell m}.
\end{equation}
\par
Substituting~\eqref{eqn:AgleC4} into~\eqref{eqn:hXgnQlerAg} and
using $g=f^{m\ell}$ gives
\begin{equation}
  \label{eqn:hXfmellnQ}
  h_X\bigl(f^{m\ell n}(Q)\bigr) 
    \le \left(C_8 r  m^r (\d_f+\e)^{\ell m} \right)^n h_X(Q).
\end{equation}
We now take~$P\in X_f(\Kbar)$ as in the statement of
the theorem, and we apply~\eqref{eqn:hXfmellnQ} to each of the points
$P,f(P),\ldots,f^{m\ell-1}(P)$ to obtain
\begin{equation}
  \label{eqn:maxhX1}
  \max_{0\le i<m\ell} h_X\bigl(f^{m\ell n+i}(P)\bigr) 
    \le  \left(C_8 r  m^r (\d_f+\e)^{\ell m} \right)^n 
        \max_{0\le i<m\ell} h_X\bigl(f^i(P)\bigr).
\end{equation}
For $0\le i<m\ell$, we apply Proposition~\ref{proposition:hDflehfD} to
each of the heights $h_X\bigl(f^i(P)\bigr)$.  Using the fact that the
ample height~$h_X$ dominates any other height~$h_D$, i.e., $h_X\gg
h_D$ with a constant depending on~$D$, we obtain
\begin{equation}
  \label{eqn:maxhX2}
  \max_{0\le i<m\ell} h_X\bigl(f^i(P)\bigr) \le C_9h_X(P).
\end{equation}
Combining~\eqref{eqn:maxhX1} and~\eqref{eqn:maxhX2} gives
\begin{equation}
  \label{eqn:maxhX3}
  \max_{0\le i<m\ell} h_X\bigl(f^{m\ell n+i}(P)\bigr) 
    \le C_9 \left(C_8 r  m^r (\d_f+\e)^{\ell m} \right)^n h_X(P).
\end{equation}
\par
Now let $q\ge1$ be any integer and write
\[
  q = m\ell n+i\quad\text{with $0\le i<m\ell$.}
\]
Then~\eqref{eqn:maxhX3} implies that
\begin{equation}
  \label{eqn:maxhX4}
  h_X\bigl(f^q(P)\bigr)
  \le C_9 (C_8rm^r)^{q/m\ell} (\d_f+\e)^q h_X(P),
\end{equation}
where we have used the trivial estimates $\ell m n\le q$ and $n\le
q/m\ell$. The key point to note about the inequality~\eqref{eqn:maxhX4}
is that the quantity $(C_8rm^r)^{1/m\ell}$ is independent of~$q$
and goes to~$1$ as $m\to\infty$. So we now fix a value of~$m$ such
that 
\[
  (C_8rm^r)^{1/m\ell} \le (1+\e).
\]
This value of~$m$ depends on~$\e$, and of course it depends on~$X$
and~$f$, but it does not depend on the integer~$q$ or the
point~$P$. We note that the constant~$C_9$ now also depends on~$\e$,
but not on~$q$ or~$P$.  Hence~\eqref{eqn:maxhX4} becomes
\begin{equation}
  \label{eqn:maxhX5}
  h_X\bigl(f^q(P)\bigr)
  \le C_9 (1+\e)^q (\d_f+\e)^q h_X(P).
\end{equation}
We have proven that~\eqref{eqn:maxhX5} holds for all~$P\in X_f(\Kbar)$
and all~$q\ge0$, where does not depend on~$q$.  After adjusting~$\e$,
the inequality~\eqref{eqn:maxhX5} is the desired result, which
completes the proof of Theorem~\ref{theorem:hXfnledfnhX}.
\end{proof}

\section{An application to canonical heights}
\label{section:nefcanht}
In this section we use Theorem~\ref{theorem:hXfnledfnhX} to prove
Theorem~\ref{theorem:hDfPbnhDfnP}, which says that the usual canonical
height limit converges for certain eigendivisor classes relative to
\emph{algebraic} equivalence. We remark that the result is well-known
(and much easier to prove) for eigendivisor classes relative to
\emph{linear} equivalence; cf.\ \cite{callsilv:htonvariety}.

\begin{proof}[Proof of Theorem~\textup{\ref{theorem:hDfPbnhDfnP}}]
To ease notation, we will let $\d=\d_f$.
\par\noindent(a)\enspace
Theorem~\ref{theorem:hXfnledfnhX} says
that for every~$\e>0$ there is a constant $C_1=C_1(X,h_X,f,\e)$ such that
\begin{equation}
  \label{eqn:hAfnPleden}
  \hplus_X\bigl(f^n(P)\bigr) \le C_1\cdot(\d+\e)^n\cdot\hplus_X(P)
  \quad\text{for all $n\ge0$.}
\end{equation}
We are given that~$f^*D\equiv\b D$. Applying Lemma~\ref{lemma:algeqht}
with $E=f^*D-\b D$, we find a constant~$C_2=C_2(D,A,f)$ such that
\begin{equation}
  \label{eqn:hfstarDQ}
  \bigl| h_{f^*D}(Q) - \b h_D(Q) \bigr| \le C_2 \sqrt{\hplus_X(Q)}
  \quad\text{for all $Q\in X(\Kbar)$.}
\end{equation}
Since we have assumed that~$f$ is a morphism, standard functoriality
of the Weil height says that $h_{f^*D}=h_D\circ f+O(1)$, 
so~\eqref{eqn:hfstarDQ} becomes
\begin{equation}
  \label{eqn:hDfQdhDQleChAQ}
  \bigl| h_D\bigl(f(Q)\bigr) - \b h_D(Q) \bigr| \le C_3 \sqrt{\hplus_X(Q)}
  \quad\text{for all $Q\in X(\Kbar)$.}
\end{equation}
\par
For $N\ge M\ge 0$ we estimate a telescoping sum,
\begin{align}
\label{eqn:MNdifleCsum}
  &\Bigl|\b^{-N}h_D\bigl(f^N(P)\bigr) -  \b^{-M}h_D\bigl(f^M(P)\bigr)\Bigr| 
          \notag\\*
  &= \biggl|\sum_{n=M+1}^N \b^{-n}\Bigl(h_D\bigl(f^n(P)\bigr)
      - \b h_D\bigl(f^{n-1}(P)\bigr)\Bigr) \biggr| \notag\\
  &\le \sum_{n=M+1}^N \b^{-n}\Bigl|h_D\bigl(f^n(P)\bigr)
      - \b h_D\bigl(f^{n-1}(P)\bigr)\Bigr| \notag\\
  &\le \sum_{n=M+1}^N \b^{-n} C_3 \sqrt{\hplus_X\bigl(f^{n-1}(P)\bigr)}
     \quad\text{applying \eqref{eqn:hDfQdhDQleChAQ} with $Q=f^{n-1}(P)$,}
            \notag\\
  &\le \sum_{n=M+1}^N \b^{-n} C_3 \sqrt{  C_1(\d+\e)^{n-1}\hplus_X(P) }
     \quad\text{from \eqref{eqn:hAfnPleden},} \notag\\*
  &\le C_4
     \sum_{n=M+1}^\infty \left(\frac{\d+\e}{\b^2}\right)^{n/2}\sqrt{\hplus_X(P)}.
\end{align}
\par
By assumption we have $\b>\sqrt{\d}$, so we can take
\[
  \e=\frac{\b^2-\d}{2},
  \quad\text{which implies that}\quad
  \g := \frac{\d+\e}{\b^2} = 1 - \frac{\b^2-\d}{2\b^2} < 1.
\]
Hence the series~\eqref{eqn:MNdifleCsum} converges, and we obtain the
estimate
\begin{equation}
  \label{eqn:bNhDfNbMhDfM}
  \Bigl|\b^{-N}h_D\bigl(f^N(P)\bigr) -  \b^{-M}h_D\bigl(f^M(P)\bigr)\Bigr| 
  \le C_5 \g^{M/2} \sqrt{\hplus_X(P)},
\end{equation}
where $C_5=C_5(X,f,D)$ is independent of~$P$,~$N$, and~$M$.
Then~\eqref{eqn:bNhDfNbMhDfM} and the fact that $\g<1$ imply that the
sequence $\b^{-n}h_D\bigl(f^n(P)\bigr)$ is Cauchy, which proves~(a).
\par\noindent(b)\enspace
The formula $\hhat_{D,f}\bigl(f(P)\bigr) = \b \hhat_{D,f}(P)$ follows
immediately from the limit defining~$\hhat_{D,f}$ in~(a).  Next,
letting~$N\to\infty$ and setting~$M=0$ in in~\eqref{eqn:bNhDfNbMhDfM}
gives
\[
  \bigl|\hhat_{f,D}(P) - h_D(P)\bigr| \le C_5 \sqrt{\hplus_X(P)},
\]
which completes the proof of~(b).
\par\noindent(c)\enspace
We are assuming that $\hhat_{f,D}(P)\ne0$.  If $\hhat_{f,D}(P)<0$, we
change~$D$ to~$-D$, so we may assume that $\hhat_{f,D}(P)>0$.
Let~$H\in\Div(X)$ be an ample divisor such that $H+D$ is also ample.
(This can always be arranged by replacing~$H$ with~$mH$ for a
sufficiently large~$m$.) Since~$H$ is ample, we may assume that the
height function~$h_H$ is non-negative.  We compute
\begin{align*}
  h_{D+H}&\bigl(f^n(P)\bigr) \\
  &= h_D\bigl(f^n(P)\bigr) + h_H\bigl(f^n(P)\bigr) + O(1) \\
  &\ge h_D\bigl(f^n(P)\bigr) + O(1)
      \quad\text{since $h_H\ge0$,} \\
  &= \hhat_{f,D}\bigl(f^n(P)\bigr) 
       + O\left(\sqrt{\hplus_X\bigl(f^n(P)\bigr)}\right)
       \quad\text{from (b),} \\
  &= \b^n\hhat_{f,D}(P)
       + O\left(\sqrt{\hplus_X\bigl(f^n(P)\bigr)}\right)
       \quad\text{from (b),} \\
  &= \b^n\hhat_{f,D}(P)
       + O\left(\sqrt{C(\d+\e)^n\hplus_X(P) }\right)
       \quad\text{from Theorem~\ref{theorem:hXfnledfnhX}.}
\end{align*}
This estimate is true for every~$\e>0$, where~$C$ depends on~$\e$.
Using the assumption that~$\b>\sqrt{\d}$, we 
can choose an~$\e>0$ satisfying $\d+\e<\b^2$. This gives
\[
  h_{D+H}\bigl(f^n(P)\bigr)
  \ge \b^n\hhat_{f,D}(P) + o(\b^n),
\]
so taking $n^{\text{th}}$-roots, using the assumption that $\hhat_{f,D}(P)>0$,
and letting~$n\to\infty$ yields
\[
  \alower_f(P)
  = \liminf_{n\to\infty} h_{D+H}\bigl(f^n(P)\bigr)^{1/n}
  \ge \b.
\]
(Note that Proposition~\ref{proposition:afPindepofht} says that we can
use~$h_{D+H}$ to compute~$\alower_f(P)$, since~$D+H$ is ample.)
\par\noindent(d)\enspace
From~(c) we get $\alower_f(P)\ge\b=\d_f$,
while Theorem~\ref{theorem:afPledfintro}
gives $\aupper_f(P)\le\d_f$.  Hence the limit defining
$\a_f(P)$ exists and is equal to~$\d_f$.
\par\noindent(e)\enspace
One direction is trivial. For the other, suppose that $\hhat_{D,f}(P)=0$.
Since we are assuming that~$D$ is ample, we may take $h_X=h_D$ and $h_D\ge1$.
Then for any~$n\ge0$, we apply~(b) to the point~$f^n(P)$ to obtain
\[
  0 = \b^n\hhat_{D,f}(P) = \hhat_{D,f}\bigl(f^n(P)\bigr)
  \ge h_D\bigl(f^n(P)\bigr) - c\sqrt{h_D\bigl(f^n(P)\bigr)}.
\]
This gives $h_D\bigl(f^n(P)\bigr) \le c^2$, where~$c$ does not depend
on~$P$ or~$n$.  This shows that~$\Orbit_f(P)$ is a set of bounded
height with respect to an ample height.  Since~$\Orbit_f(P)$ is
contained in~$X\bigl(K(P)\bigr)$ and since we have assumed that $K$ is
a number field, we conclude that~$\Orbit_f(P)$ is finite.
\end{proof}

\begin{remark}
\label{remark:fDdfDexists}
If~$f$ is a morphism, then De-Qi Zhang has pointed out that there is
always at least one nonzero nef divisor class~$D\in\NS(X)_\RR$
satisfying $f^*D \equiv \d_f D$.  So there is always at least one
nontrivial nef divisor class to which
Theorem~\ref{theorem:hDfPbnhDfnP} applies, although there need not be
any such ample divisor classes.  The existence of such a~$D$ is an
immediate consequence of the following elementary
Perron--Frobenius-type result of Birkhoff, applied to the vector
space~$\RR^r=\NS(X)_\RR$, the linear transformation~$T=f^*$, and the
cone~$C=\Nef(X)$; cf.\ \cite[Lemma~1.12]{MR1867314}.
\end{remark}

\begin{proposition}
\label{proposition:birkhoff}
\textup{(Birkhoff~\cite{MR0214605})} Let $C\subset\RR^r$ be a strictly
convex closed cone with nonempy interior, and let $T:\RR^r\to\RR^r$ be an
$\RR$-linear map with $T(C)\subseteq C$. Then~$C$ contains an
eigenvector whose eigenvalue is the spectral radius of~$T$.
\end{proposition}

\begin{question}
It would be interesting to know if Theorem~\ref{theorem:hDfPbnhDfnP}
is true for algebraically stable rational maps that are not morphisms.
\end{question}

\section{An alternative proof of Proposition~\ref{proposition:hDflehfD}}
\label{section:htinequalityaltpf}

In this section we give an alternative, more elementary, proof of
Proposition~\ref{proposition:hDflehfD}.  The proof uses three lemmas,
one geometric, one arithmetic, and the third combining the first two.

\begin{lemma}
\label{lemma:effplusamp}
Let~$D\in\Div(X/K)$ be an effective divisor. Then there exists an
integer~$r\ge1$ and an effective ample divisor~$D'\in\Div(X/K)$ such
that~$rD+D'$ is ample.
\end{lemma}
\begin{proof}
Let~$H\in\Div(X/K)$ be an ample divisor. Then there exists an
integer~$m\ge1$ such that~$mH-D$ is ample, and hence an
integer~$r\ge1$ such that~$rmH-rD$ is very ample. Since~$rmH-rD$ is
very ample, there is an effective (and necessarily very ample)
divisor~$D'$ that is linearly equivalent to $rmH-rD$. Then~$rD+D'\sim
rmH$ is (very) ample, since it is a positive multiple of a very ample
divisor
\end{proof}

\begin{lemma}
\label{lemma:htmorecoords}
Let~$\a_0,\ldots,\a_n,\b_0,\ldots,\b_m\in\Kbar$ with not all of
the~$\a_i$ equal to~$0$. Then
\[
  h\bigl([\a_0,\ldots,\a_n,\b_0,\ldots,\b_m]\bigr)
  \ge h\bigl([\a_0,\ldots,\a_n]\bigr).
\]
\end{lemma}
\begin{proof}
Extending~$K$, we may assume
that~$\a_0,\ldots,\a_n,\b_0,\ldots,\b_m\in K$.  Letting~$M_K$ be an
appropriately normalized set of inequivalent absolute values on~$K$,
the definition of the Weil height on~$\PP^n$ gives
\begin{align*}
  h\bigl([\a_0,\ldots,\a_n]\bigr)
  &= \sum_{v\in M_K} \log\max \bigl\{\|\a_0\|_v,\ldots,\|\a_n\|_v\bigr\} \\
  &\le \sum_{v\in M_K} \log\max \bigl\{\|\a_0\|_v,\ldots,\|\a_n\|_v,
      \|\b_0\|_v,\ldots,\|\b_m\|_v\bigr\} \\
  &= h\bigl([\a_0,\ldots,\a_n,\b_0,\ldots,\b_m]\bigr),
\end{align*}
which completes the proof of Lemma~\ref{lemma:htmorecoords}.
\end{proof}

\begin{lemma}
\label{lemma:hEPgehyP}
Let~$D\in\Div(X)$ be an effective divisor,
let
\[
  1=x_0,x_1,\ldots,x_n\in\G\bigl(X,\Ocal(D)\bigr),
\]
and fix a height function~$h_D$ on~$X(\Kbar)$ associated to~$D$. Then
there is a constant~$C=C(X,f,h_D)$ such that for all points~$P\in
X(\Kbar)$ such that~$x_0,\ldots,x_n$ are defined at~$P$,
\[
  h_D(P) \ge h\bigl(\bigl[x_0(P),x_1(P),\ldots,x_n(P)\bigr]\bigr) - C.
\]
\end{lemma}
\begin{proof}
Let
\[
  \t=[x_0,\ldots,x_n] : X \dashrightarrow \PP^n
\]
be the rational map induced by the functions~$x_0,\ldots,x_n$.
\par
We first prove that it suffices to prove the lemma for a positive
multiple~$dD$ of~$D$. We use the $d$-uple 
embedding~$\s_d : \PP^n\to\PP^N$; see~\cite[Exercise~I.2.12]{hartshorne}.
The~$d$-uple embedding has the property that there is an exact
equality~\cite[Proposition~B.2.4]{hindrysilverman:diophantinegeometry}
\begin{equation}
  \label{eqn:hduple}
  h\bigl(\s_d(Q)\bigr) = dh(Q).
\end{equation}
Suppose that the lemma is true for~$dD$ and all choices of functions
in~$\G\bigl(X,\Ocal(dD)\bigr)$. We take the functions $y_0,\ldots,y_m$
consisting of all monomials~$x_0^{e_0}x_1^{e_1}\cdots x_n^{e_n}$ satisfying
$e_i\ge0$ and $\sum e_i=d$. We note that every~$y_i$ is
in~$\G\bigl(X,\Ocal(dD)\bigr)$. Then
\begin{align*}
  h_D(P)
  &= \frac{1}{d}h_{dD}(P) \\*
  &\ge \frac{1}{d} h\bigl(\bigl[y_0(P),y_1(P),\ldots,y_m(P)\bigr]\bigr) - C \\*
  &\omit\hfill\hspace{3em}
     since we are assuming that the lemma is true for $dD$, \\
  &= \frac{1}{d} h\bigl(\s_d(\t(P))\bigr) - C \\
  &= h\bigl(\t(P)\bigr) - C 
     \quad\text{from \eqref{eqn:hduple},} \\*
  &= h\bigl(\bigl[x_0(P),x_1(P),\ldots,x_n(P)\bigr]\bigr) - C.
\end{align*}
\par
We use Lemma~\ref{lemma:effplusamp} to find an integer~$r\ge1$ and an
effective ample divisor~$D'\in\Div(X/k)$ such that~$rD+D'$ is ample.
As noted above, we may replace~$D$ by~$rD$, and by the same remark, we
may replace~$D$ and~$D'$ by appropriate multiples so that~$D'$
and~$D+D'$ are very ample.  We choose a
basis~$1=z_0,z_1,\ldots,z_\ell$ for~$\G\bigl(X,\Ocal_X(D')\bigr)$.
Then the functions~$x_iz_j$ satisfy
\[
  x_iz_j\in\G\bigl(X,\Ocal_X(D+D')\bigr)
  \quad\text{for $0\le i\le n$ and $0\le j\le \ell$,}
\]
so we can find a spanning set $1=w_0,w_1,\ldots,w_k$
for~$\G\bigl(X,\Ocal_X(D+D')\bigr)$ whose first $(n+1)(\ell+1)$
elements are the functions~$x_iz_j$.
\par
In order to define the Weil height associated to a divisor, one writes
the divisor as the difference of very ample divisors and takes the
difference of the heights, where the height associated to a very ample
divisor is defined by using an associated projective embedding. In
our case, we have written~$D$ as the difference $(D+D')-D'$, so we
have
\begin{align*}
  h_D(P)
  &= h_{D+D'}(P) - h_{D'}(P) \\*
  &= h\bigl(\bigl[w_0(P),\ldots,w_k(P)\bigr]\bigr)
       - h\bigl(\bigl[z_0(P),\ldots,z_\ell(P)\bigr]\bigr) \\
  &\ge h\bigl(\bigl[x_iz_j(P)\bigr]_{0\le i\le n,\,0\le j\le\ell}\bigr)
       - h\bigl(\bigl[z_0(P),\ldots,z_\ell(P)\bigr]\bigr) \\*
  &\omit\hfill from Lemma~\ref{lemma:htmorecoords}, \\*
  &= h\bigl(\bigl[x_0(P),\ldots,x_n(P)\bigr]\bigr) 
    \quad\text{from
      \cite[Proposition~B.2.4(b)]{hindrysilverman:diophantinegeometry}}\\*
  &\omit\hfill 
       (Segre embedding).
\end{align*}
Choosing a different representative for~$h_D$ will introduce a bounded
error, which accounts for the~$C$ in the statement of the lemma.
This completes the proof of Lemma~\ref{lemma:hEPgehyP} for all
points at which the 
functions $x_0,\ldots,x_n,w_0,\ldots,w_k,z_0,\ldots,z_\ell$ are regular. But
since~$D+D'$ and~$D'$ are very ample, we can repeat the argument using
a finite number of other bases for~$\G\bigl(X,\Ocal_X(D+D')\bigr)$
and~$\G\bigl(X,\Ocal_X(D')\bigr)$ so as to obtain the desired estimate
for all points at which~$x_0,\ldots,x_n$ are regular. 
\end{proof}

\begin{proof}[Alternative Proof of Proposition~$\ref{proposition:hDflehfD}$]
Replacing~$D$ by a multiple, we may assume that~$D$ is very ample
and effective. We let~$1=x_0,x_1,\ldots,x_n$ be a basis for
$\G\bigl(X,\Ocal_X(D)\bigr)$.
\par
Let~$E\in\Div(X)$ be a prime divisor, i.e., an irreducible
codimension~$1$ subvariety of~$X$. Then by definition~$f^*E$ is equal
to the Zariski closure $\overline{f^{-1}(E\setminus I_f)}$. Hence our
assumption that~$D$ is effective implies that~$f^*D$ is effective.
Further, there is a natural map
\[
  f^* : \G\bigl(X,\Ocal_X(D)\bigr)
  \longrightarrow \G\bigl(Y,\Ocal_X(f^*D)\bigr),
\]
so in particular,
\[
  f^*x_0,\ldots,f^*x_n \in \G\bigl(Y,\Ocal_X(f^*D)\bigr).
\]
We apply Lemma~\ref{lemma:hEPgehyP} to the divisor~$f^*D$ and
functions $f^*x_0,\ldots,f^*x_n$. This yields
\begin{equation}
  \label{eqn:hfDPge}
  h_{Y,f^*D}(P) \ge h\bigl(\bigl[f^*x_0(P),\ldots,f^*x_n(P)\bigr]\bigr) - C.
\end{equation}
On the other hand, the functions~$x_0,\ldots,x_n$ give an embedding
\[
  \t=[x_0,\ldots,x_n] : X \hookrightarrow \PP^n
  \quad\text{satisfying}\quad \t^*\Ocal_{\PP^n}(1)=\Ocal_X(D),
\]
so for points~$Q\in X(\Kbar)$ at which~$x_0,\ldots,x_n$ are regular,
we have
\[
  h_{X,D}(Q) = h\bigl(\t(Q)\bigr)
  = h\bigl(\bigl[x_0(Q),x_1(Q),\ldots,x_n(Q)\bigr]\bigr) + O(1).
\]
Applying this with $Q=f(P)$ and noting that
$x_i\bigl(f(P)\bigr)=f^*x_i(P)$, we find that
\begin{equation}
  \label{eqn:hDfPeqhfx}
  h_{X,D}\bigl(f(P)\bigr)
  = h\bigl(\bigl[f^*x_0(P),\ldots,f^*x_n(P)\bigr]\bigr) + O(1).
\end{equation}
Combining~\eqref{eqn:hfDPge} and~\eqref{eqn:hDfPeqhfx} gives
\[
  h_{Y,f^*D}(P) \ge  h_{X,D}\bigl(f(P)\bigr) + O(1),
\]
which gives the desired result for points where all of the
functions $f^*x_0,\ldots,f^*x_n$ are regular.  By taking a finite
number of different effective divisors in the very ample divisor class
of~$D$, we obtain analogous inequalities that cover all points~$P$ at
which~$f$ is defined.
\end{proof}

\section{Some Instances of Conjecture~\ref{conjecture:afPeqdfY}}
\label{section:caseafPeqdf}

Let~$P\in X_f(\Kbar)$.  We recall that
Conjecture~\ref{conjecture:afPeqdfY} asserts:
\begin{parts}
\Part{\textbullet}
$\a_f(P)$ exists and is an algebraic integer.
\Part{\textbullet}
$\bigl\{\a_f(P):P\in X_f(\Kbar)\bigr\}$ is a finite set.
\Part{\textbullet}
If~$\Orbit_f(P)$ is Zariski dense in~$X$, then $\a_f(P)=\d_f$.
\end{parts}

The following theorem describes some cases for which we can prove
Conjecture~\ref{conjecture:afPeqdfY}.

\begin{theorem}
\label{theorem:casesconjistrue}
Conjecture~$\ref{conjecture:afPeqdfY}$ is true in the following
situations\textup:
\begin{parts}
\Part{(a)}
$f$ is a morphism and $\NS(X)_\RR=\RR$. 
\Part{(b)}
$f:\PP^N\dashrightarrow\PP^N$ extends a regular affine
automorphism~$\AA^N\to\AA^N$.
\Part{(c)}
$X$ is a smooth projective surface and $f$ is an automorphism.
\Part{(d)}
$f:\PP^N\dashrightarrow\PP^N$ is a  monomial map and~$P\in\GG_m^N(\Kbar)$.
\Part{(e)}
$X$ is an abelian variety and $f:X\to X$ is an endomorphism.
\end{parts}
\end{theorem}
\begin{proof}
See~\cite{kawsilv:dfeqafPexamples} for~(a,b,c),
see~\cite{arxiv1111.5664} for~(d),
and see~\cite{kawsilv:jordanblock} for~(e).
\end{proof}

\begin{remark}
The maps in Theorem~\ref{theorem:casesconjistrue}(a,b,c) are
algebraically stable. (This is automatic for morphisms, and it is also
true for regular affine automorphisms.)
We note that if~$f$ is algebraically stable, then
\[
  \d_f
   = \lim_{n\to\infty} \rho\bigl((f^n)^*\bigr)^{1/n}
   = \lim_{n\to\infty} \rho\bigl((f^*)^n\bigr)^{1/n}
   = \rho(f^*),
\]
so~$\d_f$ is automatically an algebraic integer.  Monomial maps are
not, in general, algebraically stable, but their dynamical degrees are
known to be algebraic integers~\cite{MR2358970}.
\end{remark}

We also mention the following result
from~\cite{kawsilv:dfeqafPexamples} which shows in certain cases
that~$\a_f(P)=\d_f$ for a ``large'' collection of points. The proof 
uses $p$-adic methods, weak lower canonical heights, and Guedj's
classification of degree~$2$ planar maps~\cite{MR2097402}.

\begin{theorem}
\label{theorem:mainthma}
Let $f:\AA^2\to\AA^2$ be an affine morphism defined over~$\Kbar$ whose
extension to $f:\PP^2\dashrightarrow\PP^2$ is dominant.
Assume that either of the following is true\textup{:}
\begin{parts}
\Part{(a)}
The map~$f$ is algebraically stable.
\Part{(b)}
$\deg(f)=2$.
\end{parts}
Then
\[
  \bigl\{P\in \AA^2(\Kbar) : \a_f(P)=\d_f\bigr\}
\]
contains a Zariski dense set of points having disjoint orbits.
\end{theorem}
\begin{proof}
See~\cite{kawsilv:dfeqafPexamples}.
\end{proof}


\def\cprime{$'$}

\appendix

\section{Proof of Lemma~$\ref{lemma:gShS0inf}$}
\label{appendix:hgxlemma}
In this section we prove Lemma~\ref{lemma:gShS0inf}, which we restate
for the convenience of the reader:

\begin{unnumberedlemma}
\textup{(Lemma~\ref{lemma:gShS0inf})}
Let~$S$ be a set, let $g:S\to S$ and $h:S\to[0,\infty)$ be maps, and let
$a\ge1$ and $c\ge1$ be constants. Suppose that for all $x\in S$ we have
\begin{equation}
  \label{eqn:xhgzleahxcsq}
  h\bigl(g(x)\bigr) \le a h(x) + c\sqrt{h(x)}.
\end{equation}
Then for all $x\in S$ and all $n\ge0$,
\begin{equation}
  \label{eqn:xhgnxlean}
  h\bigl(g^n(x)\bigr)
  \le a^n \left(  h(x) + (2\sqrt2 c)^n\sqrt{h(x)} \right).
\end{equation}
\end{unnumberedlemma}

\begin{proof}[Proof of Lemma~$\ref{lemma:gShS0inf}$]
To ease notation, we let $\g=2\sqrt2$.
The proof is by induction on~$n$. 
The inequality~\eqref{eqn:xhgnxlean} is trivially true for~$n=0$,
and for~$n=1$, the desired inequalty~\eqref{eqn:xhgnxlean} is weaker
than the assumed estimate~\eqref{eqn:xhgzleahxcsq}.
Suppose now that~\eqref{eqn:xhgnxlean} is true for~$n$. Then
\begin{align*}
  h\bigl(g^{n+1}(x)\bigr)
  &= h\bigl(g^n(g(x))\bigr) \\
  &\le a^n \left(  h(g(x)) + (\g c)^n\sqrt{h(g(x))} \right) \\
  &\omit\hfill from the induction hypothesis,\\
  &\le a^n \left( a h(x) + c\sqrt{h(x)}
      + (\g c)^n\sqrt{a h(x) + c\sqrt{h(x)}} \right) \\
  &\omit\hfill from \eqref{eqn:xhgzleahxcsq}, \\
  &\le a^n \left( a h(x) + c\sqrt{h(x)}
      + (\g c)^n\sqrt{2ac h(x)} \right) \\
  &\omit\hfill since $a,c,h(x)\ge1$ \\
  &= a^{n+1} h(x) + \left(a^nc + (\g a c)^n\sqrt{2ac}\right)\sqrt{h(x)}.
\end{align*}
Hence
\begin{align*}
  a^{n+1} &\left(  h(x) + (\g c)^{n+1}\sqrt{h(x)} \right)
                  -   h\bigl(g^{n+1}(x)\bigr) \\*
  &\ge  \Bigl( a^{n+1}  h(x) + (\g a c)^{n+1}\sqrt{h(x)} \Bigr) \\*
  &\hspace{6em}{}
    -  \Bigl( a^{n+1} h(x)
           + \left(a^nc + (\g a c)^n\sqrt{2ac}\right)\sqrt{h(x)} \Bigr) \\
  &= \sqrt{h(x)} a^n c 
        \Bigl( \g^{n+1} a c^n - 1 - \g^n a^{1/2} c^{n-1/2}\sqrt2\Bigr) \\
  &\ge \sqrt{h(x)} a^n c 
        \Bigl( \g^{n+1} a c^n - 1 - \g^n a c^n \sqrt2\Bigr) \\
  &= \sqrt{h(x)} a^n c \Bigl( \g^n a c^n(\g-\sqrt2) - 1 \Bigr) \\
  &= \sqrt{h(x)} a^n c \Bigl( \g^n a c^n\sqrt2 - 1 \Bigr) 
    \quad\text{since $\g=2\sqrt2$,}\\*
  &> 0
    \quad\text{since $a,c\ge1$.}
\end{align*}
This shows that~\eqref{eqn:xhgnxlean} is true for~$n+1$, which
completes the proof of the lemma.
\end{proof}

\end{document}